\documentclass[12pt]{article}

\usepackage{times}

\usepackage[a4paper,text={128mm,185mm}, centering]{geometry}

\usepackage{changepage}

\usepackage{titlesec}

\titleformat{\section}[hang]%
{\bfseries\large}{\thesection.}{1ex}{}%

\titleformat{\subsection}[hang]%
{\bfseries}{\thesubsection}{1ex}{}%

\usepackage{amsthm}
\usepackage{fancyhdr}
\theoremstyle{plain}
\newtheorem{thm}{Theorem}[section]

\newtheorem{cor}[thm]{Corollary}
\newtheorem{lem}[thm]{Lemma}
\newtheorem{prop}[thm]{Proposition}
\newtheorem{defn}[thm]{Definition}

\theoremstyle{remark}
\newtheorem{rmk}[thm]{Remark}
\newtheorem{ex}[thm]{Example} 

\newtheorem{constr}[thm]{Construction} 
\newtheorem{conventions}[thm]{Conventions}
\newtheorem*{thm*}{Theorem}

\usepackage{hyperref} 
\hypersetup{colorlinks=true, linkcolor=black, citecolor=black}  
\usepackage{amsmath, amsfonts, amssymb}
\usepackage{bbm, tikz-cd, tikz}
\tikzset{shorten <>/.style={shorten >=#1,shorten <=#1}}
\usepackage{tkz-euclide}
\usetikzlibrary{decorations.markings,arrows}

\usepackage{array}
\usepackage{float}
\usepackage{multirow}       
\newcolumntype{C}{>{$ \displaystyle } c <{$}} 

\usepackage[style=numeric, backend=biber]{biblatex}

\usepackage{microtype}

\newcommand{\Fun}{{\sf {Fun}}}

\newcommand{\Bibun}{{\sf {Bibun}}}

\newcommand{\sq}{/\!\!/}
\newcommand{\id}{{{\rm id}}}
\newcommand{\Hom}{{\sf Hom}}
\newcommand{\Aut}{{\sf Aut}}
\newcommand{\Bun}{{\sf Bun}}
\newcommand{\Stab}{{\rm Stab}}
\newcommand\EquivTo{\xrightarrow{
		\,\smash{\raisebox{-0.65ex}{\ensuremath{\scriptstyle\sim}}}\,}}

\newcommand{\fl}{{\sf fl}}
\newcommand{\cG}{\mathcal{G}}
\newcommand{\cP}{\mathcal{P}}
\newcommand{\cS}{\mathcal{S}}
\newcommand{\cY}{\mathcal{Y}}
\newcommand{\bU}{\mathbf{U}}
\newcommand{\Act}{{\sf Act}}

\newcommand{\one}{\mathbbm{1}}
\newcommand{\pitwoS}{\pi_{\le 2}(S^2)}
\newcommand{\Hol}{\mathsf{Hol}}

\newcommand{\ptG}{*_{\cG}}
\newcommand{\sigmaG}{g}
\newcommand{\tauG}{g'}
\newcommand{\fcomp}{f_{\id_*, \id_*}}
\newcommand{\ecomp}{e_{\id_*, \id_*}}
\newcommand{\dcomp}{d_{\id_*, \id_*}}

\newcommand{\nonempty}[1]{I^#1_{\ne \varnothing}}

\newcommand{\pr}[1]{\left( #1 \right)}              
\newcommand{\bpr}[1]{\bigl( #1 \bigr)}              
\newcommand{\gam}[1]{\gamma_{#1}}

\title{\vskip 5pt  \bf  FLAT PRINCIPAL 2-GROUP BUNDLES ON $S^2$ AND THEIR HOLONOMY}
\author{\itshape\bfseries { Emily CLIFF, Yves LOLELO, Sandrine ST-CYR}}
\date{October 2026}

\begin{document}
	\maketitle
	
	\cfoot{}
	\thispagestyle{empty}
	\vskip 25pt
	\begin{adjustwidth}{0.5cm}{0.5cm}
		{\small
			{\bf R\'esum\'e.} Soit $\cG(G,A,\alpha)$ un 2-groupe discret, donn\'e comme l'extension cat\'egorique d'un groupe ordinaire $G$ par un groupe ab\'elien $A$. Nous d\'efinissons les fibr\'es $\cG$-principaux sur une vari\'et\'e en termes de donn\'ees de \v {C}ech par rapport \`a un bon recouvrement d'ouverts. Dans le cas de la sph\`ere $S^2$, nous calculons sa bicat\'egorie explicitement et identifions son 2-squelette. Nous la comparons avec la bicat\'egorie d'holonomie $\Fun(* \sq \pitwoS, * \sq \cG)$, et construisons une \'equivalence entre celles-ci.  \\
			{\bf Abstract.} Let $\cG(G,A,\alpha)$ be a discrete 2-group, given as the {categorical} extension of an ordinary group $G$ by an abelian group $A$. We define principal $\cG$-bundles over a manifold in terms of \v{C}ech data over a good open cover. In the case of the manifold $S^2$, we calculate this bicategory explicitly, and identify its 2-skeleton. We compare with the holonomy bicategory $\Fun(* \sq \pitwoS, * \sq \cG)$, and provide a concrete equivalence between the two. 
		}
	\end{adjustwidth}
	
	\section{Introduction}
	Let $G$ be a Lie group and let $X$ be a smooth connected manifold. A well-known result classifies flat principal $G$-bundles on $X$ by their holonomy: more specifically, taking the holonomy of a flat principal $G$-bundle around homotopy classes of loops yields an equivalence of categories
	\begin{align}\label{eq: classical holonomy equivalence}
		\Bun_G^\fl(X) \EquivTo \Fun(*\sq \pi_1(X), *\sq G).
	\end{align}
	
	This paper concerns the generalization of this result to the setting of flat principal bundles for a 2-group $\cG$.  Indeed, it is widely expected that for a discrete\footnote{Since we are interested only in flat bundles in this paper, there is no harm in assuming our 2-group is discrete. On the other hand, once we make this assumption, all bundles are automatically flat. Therefore, we suppress the decoration $\fl$ from our notation.} 2-group $\cG$, the equivalence \eqref{eq: classical holonomy equivalence} is replaced by an equivalence of bicategories 
	\begin{align}\label{eq: 2-group holonomy equivalence}
		\Bun_\cG(X) \EquivTo \Fun(* \sq\pi_{\le 2}(X), *\sq \cG),
	\end{align}
	where $\pi_{\le 2}(X)$ is the \emph{fundamental 2-group} of the manifold $X$, and the right-hand side is the bicategory of functors between one-object bicategories. Our main theorem is the following special case:
	
	\begin{thm}\label{thm: main}
		Let $\cG$ be a discrete 2-group. Then we have an equivalence of bicategories
		\begin{align*}
			\Hol \colon \Bun_\cG(S^2) \EquivTo \Fun(* \sq\pi_{\le 2}(S^2), *\sq \cG).
		\end{align*}
	\end{thm}
	
	Other particular versions of the proposed equivalence \eqref{eq: 2-group holonomy equivalence} are already known. For example, results of Schreiber--Waldorf \cite{SchreiberWaldorf2013} and Abad--Vásquez \cite{Abad-Vasquez} imply versions of this result in the case that $\cG$ is a strict 2-group. While it is true that every 2-group is equivalent to a strict 2-group \cite{Hoang-strict}, the strictification procedure requires the addition of many objects to force associativity. Hence, while it is possible to replace a general 2-group $\cG$ in \eqref{eq: 2-group holonomy equivalence} by its strictification on both sides, and obtain a commutative square that forces the equivalence to be true, it would not be obvious how to directly write down the functor that computes the holonomy of a given $\cG$-bundle. 
	
	Furthermore, we have several motivations for wishing to work directly with non-strict 2-groups rather than replacing them with a strictification. First, in its natural form, the fundamental 2-group $\pi_{\le 2}(X)$ is generally not strict, so we already have a non-strict 2-group appearing in our holonomy bicategory. Second, while it is true that any 2-group is equivalent to a strict 2-group, it is \emph{not} true that any homomorphism between 2-groups is equivalent to a strict homomorphism, so that we need to be able to handle weak categorical structures and coherence morphisms in any case. 
	
	
	In the setting of a not necessarily strict 2-group, Berwick-Evans--Cliff--Murray \cite[Proposition 2.37]{BCM} prove the result \eqref{eq: 2-group holonomy equivalence} in the case of $X$ a $K(\pi, 1)$. This is exactly the case where the fundamental 2-group $\pi_{\le 2}(X)$ is equivalent to the ordinary fundamental group $\pi_1(X)$, so that the bicategory on the right-hand side simplifies considerably. The authors are therefore able to use techniques that do not generalize at all to the setting of a genuinely categorical fundamental 2-group. The two-dimensional sphere is the simplest manifold with non-trivial $\pi_2$, making it the canonical toy model for studying higher holonomy.
	
	We note that, the result of \cite{BCM} covers in particular all orientable closed surfaces of strictly positive genus, so Theorem \ref{thm: main} completes the classification \eqref{eq: 2-group holonomy equivalence} of principal bundles via holonomy for orientable closed surfaces.
	
	Analogously to the classical equivalence \eqref{eq: classical holonomy equivalence}, the geometric motivation underlying the expected equivalence \eqref{eq: 2-group holonomy equivalence} comes from \emph{higher parallel transport}. 
	Various authors provide definitions of connections and parallel transport for bundles for strict Lie 2-groups (see for example \cite{Baez-Huerta, Waldorf-parallel-transport}). In the context of smooth 2-groups given as categorical central extensions, Waldorf also proposes a definition of connection, in terms of trivializations of the so-called Chern--Simons 2-gerbe \cite{Waldorf-string-cs}. It remains to develop the theory of parallel transport in this setting. One expects that then the equivalence \eqref{eq: 2-group holonomy equivalence} is an enhancement of the resulting \emph{higher holonomy} construction to a bifunctor, in the case of flat connections. We anticipate that the explicit descriptions of the bicategories of both principal bundles and holonomy data provided here for the case of the sphere will be useful for exploring these constructions in a concrete, hands-on example. 
	
	We comment now on the structure of the proof, which proceeds in several steps. Since we are only interested in flat principal bundles, we may assume that our 2-group is discrete, and is equivalent to one of the form $\cG(G,A,\alpha)$ for $G$ an ordinary group, $A$ an abelian group equipped with an action of $A$, and $\alpha \colon G^3 \to A$. In Section \ref{sec: 2-groups}, we introduce this 2-group, and use the classification of  \cite{BCMNP} to define principal 2-group bundles in terms of \v{C}ech data or transition functors (see also \cite{Bartels}). In Section \ref{sec: Sandrine}, we focus on the case of $S^2$. We work with a specific good cover $\mathcal{U}$ given by four open subsets, and we are able to explicitly identify a 2-skeleton of the bicategory $\Bun_\cG(S^2)$. Indeed, we define a skeletal 2-category $\mathcal{S}$ whose objects are in bijection with the quotient $A/G$, whose 1-morphisms are given by stabilizer subgroups for the action of $G$ on $A$, and whose 2-morphisms are given by copies of $A$. Then we prove:
	\begin{prop}\label{prop: skeleton of Bun}
		There is an evident fully faithful inclusion of $\mathcal{S}$ into $\Bun_\cG(S^2)$ which is an equivalence of bicategories. Thus, $\mathcal{S}$ is a 2-skeleton of $\Bun_\cG(S^2)$. 
	\end{prop}
	
	In Section \ref{sec: Yves}, we turn our attention to the holonomy side of the correspondence. Using that the fundamental 2-group $\pi_{\le 2}(S^2)$ is equivalent to $* \sq \mathbb{Z}$, we are also able to give an explicit description of the holonomy bicategory $\Fun(* \sq\pi_{\le 2}(S^2), *\sq \cG)$, in terms of a bicategory $\mathcal{Y}$. It is easy to compare the bicategories $\mathcal{S}$ and $\mathcal{Y}$, which we do in Section \ref{sec: proof}:
	
	\begin{prop}
		There is a natural functor $F: \mathcal{S} \to \mathcal{Y}$ which is an equivalence of bicategories. 
	\end{prop}
	
	The main theorem then follows. Thus, not only do our results demonstrate that flat $\cG$-bundles on $S^2$ are indeed classified by their holonomy, they also provide two simple, explicit, and convenient alternative descriptions of these equivalent bicategories, namely $\mathcal{S}$ and $\mathcal{Y}$. 
	
	\begin{conventions}
		We use the term \emph{bicategory} for the higher categorical structure with objects, 1-morphisms, and 2-morphisms, where composition of 1-morphisms is only weakly associative, up to specified associator. (Some sources use the term \emph{weak 2-category}.) We use the notation $\vee$ and $\wedge$ for vertical and horizontal compositions of 2-morphisms, respectively. We use the term \emph{functor} for the notion of functor between two bicategories which is only required to preserve composition weakly, up to specified 2-isomorphisms. We note that all of our bicategories have only invertible 2-morphisms, so it makes no difference whether we formally require our coherence 2-morphisms to be isomorphisms or not. (Some sources use the terms \emph{lax 2-functor} or \emph{pseudofunctor}.) We also use the term \emph{natural transformation} for the weak notion of morphism between functors. Finally, \emph{modifications} are morphisms between natural transformations. 
	\end{conventions}
	
	\subsection{Acknowledgements}
	
	This article combines some of the results of the MSc theses of the second and third authors \cite{Lolelo2024, St-Cyr2026}. We thank the Département de mathématiques and the Faculté des sciences of the Université de Sherbrooke, as well as the Institut des sciences mathématiques (ISM) for supporting them during their studies. We were supported in part by the Natural Sciences and Engineering Research Council of Canada under Discovery Grant RGPIN-2022-04104 and the Fonds de Recherche du Québec (Nature et Technologie) Research Support for New Academics under Grant 327706.
	
	We thank Dan Berwick-Evans, Juan Carlos Bustamante, Toby Caouette, John Huerta, and Laura Murray for helpful discussions. 
	
	\section{\texorpdfstring{$2$-groups and their principal bundles}{2-groups and their principal bundles}}\label{sec: 2-groups}
	\subsection{Background on 2-groups}
	\begin{defn}\label{defn: 2-group}
		A (discrete) \emph{2-group} is a monoidal groupoid $(\cG,\otimes,\one)$ where every object is (weakly) $\otimes$-invertible, meaning for every object $x$ there exists an object $x^{-1}$ and isomorphisms $x\otimes x^{-1}\simeq \one\simeq x^{-1}\otimes x.$ A \emph{1-homomorphism} between 2-groups is a lax monoidal functor.
		A \emph{2-homomorphism} between 1-homomorphisms is a lax monoidal transformation.
	\end{defn}
	
	\begin{ex}\label{ex: ordinary group}
		Let $G$ be an ordinary group. We view it as a monoidal category with only identity morphisms, monoidal product induced from the group operation, and monoidal unit given by the identity element $1_G$. Since all group elements have inverses, we obtain a 2-group. 
	\end{ex}
	\begin{ex}\label{ex: BA}
		Let $A$ be an abelian group. We form the category $* \sq A$ with a single object $*$, with automorphisms given by $A$. We equip this category with a monoidal structure by setting $* \otimes * = *$ on the level of objects, and $a \otimes b = ab$ on the level of morphisms, using the group structure of $A$. This makes $* \sq A$ into a 2-group. We note that the interchange law (that is, the compatibility between tensor products and compositions of morphisms) holds here because of our assumption that $A$ is abelian. 
	\end{ex}
	
	\begin{ex}\label{ex: G, A, alpha}
		Let $G$ be an ordinary group, and let $A$ be an abelian group, denoted multiplicatively, on which $G$ acts by homomorphisms. Let $\alpha\colon G^3 \to A$ be a normalized 3-cocycle: that is, $\alpha$ is required to satisfy the following two conditions for all $g,h,k,l \in G$:
		\begin{align}
			\frac{(g \cdot \alpha(h, k, l)) \alpha(g, hk, l) \alpha(g,h,k)   }
			{\alpha(gh, k, l)\alpha(g,h,kl)} & = 1_A; \\
			\alpha(1_G,h,k) = \alpha(g,1_G,k) = \alpha(g,h,1_G) & = 1_A.
		\end{align}
		Using this data, we will now define a 2-group $\cG= \cG(G, A, \alpha)$.
		
		The objects of $\cG$ are given by the elements of $G$. The morphisms between two objects $g,h \in G$ are given as follows:
		\begin{align*}
			\Hom_\cG(g,h) = \left\{ \begin{array}{ll}
				\varnothing &  \text{if } g \ne h\\
				\{g\} \times A & \text{if } g = h.
			\end{array}
			\right.
		\end{align*}
		In other words, for fixed $g \in G$, each element $a \in A$ defines an automorphism of $g$ in $\cG$, which we will denote by $[a]_g \colon g \to g$. We will often suppress the subscript $g$ if it is clear from context. We note that in particular, $\cG$ is skeletal. Composition of morphisms is defined by the group operation of $A$: 
		\begin{align*}
			[a]_g \circ [b]_g = [ab]_g.
		\end{align*}
		It follows from the axioms of a group that composition is associative, with identity morphisms $[1_A]_g$, and that all morphisms are invertible, with $[a]_g^{-1}= [a^{-1}]_g$. 
		
		We will now define the monoidal structure on $\cG$. On objects $g,h \in G$, we define $g \otimes h = gh$. Given morphisms $[a]_g \colon g \to g$ and $[b]_h \colon h \to h$, we define $[a]_g \otimes [b]_h =[a(g\cdot b)]_{gh}$. We take the monoidal unit to be $1_G$, along with the identity unitor maps
		\begin{align*}
			l_g = [1_A]_g \colon g \otimes 1_G \to g; \quad r_g = [1_A]_g \colon 1_G \otimes g \to g. 
		\end{align*}
		Finally, we define the associator by
		\begin{align*}
			\alpha_{g,h,k} = [\alpha(g,h,k)]_{ghk} \colon (g \otimes h) \otimes k \to g \otimes(h \otimes k) \text{ for all } g,h,k \in G. 
		\end{align*}
		It is not hard to see that these structure maps satisfy the axioms of a monoidal category. In particular, the pentagon axiom follows from the cocycle condition. We can also see that each object $g$ admits a (weak) inverse $g^{-1}$, so that $\cG$ is indeed a 2-group.
	\end{ex}
	\begin{rmk}
		The 2-group $\cG(G,A,\alpha)$ is known as a \emph{categorical extension} of $G$ by $A$. In the case that the $G$-action on $A$ is trivial, it is a categorical \emph{central} extension. For instance, the ordinary group $G$ of Example \ref{ex: ordinary group} and the quotient $* \sq A$ of Example $\ref{ex: BA}$ are both categorical central extensions with trivial associators, given by $G= \cG(G, \{1_A\}, \alpha_\text{triv})$ and $ * \sq A = \cG(\{1_G\}, A, \alpha_\text{triv})$, respectively. In her thesis \cite{Hoang}, Ho\`ang proves that all 2-groups are isomorphic (i.e. weakly monoidally equivalent) to categorical extensions (see also \cite[Theorem 43]{Baez-Lauda}). Therefore, we work only with such 2-groups in this article. 
	\end{rmk}
	\begin{defn}\label{defn: fundamental 2-group}
		Let $X$ be a path-connected topological space and fix $x \in X$. The \emph{fundamental 2-group} $\pi_{\le 2}(X)$ of $X$ is the category whose objects are loops in $X$ based at $x$ and whose morphisms are homotopy classes of homotopies bewteen loops. This category has a monoidal structure given by concatenation of loops, making it into a 2-group. 
	\end{defn}
	\begin{rmk}
		In terms of Hoàng's classification \cite{Hoang,Baez-Lauda}, $\pi_{\le 2}(X)$ is equivalent to a 2-group $\cG(\pi_1(X,x), \pi_2(X), \alpha)$, where $\pi_1(X,x)$ acts on $\pi_2(X)$ by conjugation, and $\alpha$ is a choice of homotopy between concatenations $(\gamma_1 \star \gamma_2) \star \gamma_3$ and $\gamma_1 \star (\gamma_2 \star \gamma_3)$.
	\end{rmk}
	\begin{ex}
		In the case of $X= S^2$, we have $\pi_1(X) = \{1\}$ and $\pi_2(X) = \mathbb{Z}$ (here we choose to use additive notation). The action of $\pi_1(X)$ on $\pi_2(X)$ and the associator are both automatically trivial. 
	\end{ex}
	
	\subsection{\texorpdfstring{Principal $\cG$-bundles via \v Cech data}{Principal bundles via Cech data}}
	
	Let us fix (for the remainder of the paper) a 2-group of the form $\cG = \cG(G,A, \alpha)$ as in Example \ref{ex: G, A, alpha}. Let $X$ be a smooth manifold. Recall that in the classical setting of a Lie group $G$, a principal $G$-bundle $P \to X$ is classified by its transition functions or \v{C}ech data. Motivated by this, we define the bicategory of principal $\cG$-bundles on $X$ in terms of \v{C}ech data. For interested readers, in Appendix \ref{sec: appendix compare}, we will provide an alternative definition of principal $\cG$-bundles following \cite{Schommer-Pries}, and we will sketch an equivalence between the two approaches. 
	
	\begin{defn}
		An open cover $\bU = \{U_i\}_{i \in I}$ of $X$ is called a \emph{good cover} if each of the opens $U_i$ as well as each non-empty finite intersection $U_{i_1i_2\ldots i_n} = \bigcap_{j=1}^n U_{i_j}$ is contractible.
	\end{defn}
	
	\begin{rmk}
		Every smooth manifold $X$ admits a good cover; see for example the construction on page 120--121 of \cite{Weil}.
	\end{rmk}
	
	\begin{ex}\label{ex: cover}
		The sphere $S^2$ admits a good cover with four open pieces $U_1, U_2, U_3, U_4$ (see Figure \ref{fig:tetraedre}). This is given by inscribing a tetrahedron in the sphere, projecting its faces $F_i$ onto the sphere, and taking an open neighbourhood $U_i$ of each of these faces $F_i$. It is important that the open neighbourhoods be small enough that the quadruple intersection $U_{1234}$ is empty: for example, if one takes $V_i$ to be the open set given by the complement of the vertex opposite the face $F_i$, then the intersection $V_{1234}$ will be the four-times punctured sphere, which is neither empty nor contractible. 
		\begin{figure}[h!]
			\centering
			\includegraphics[width=1\textwidth]{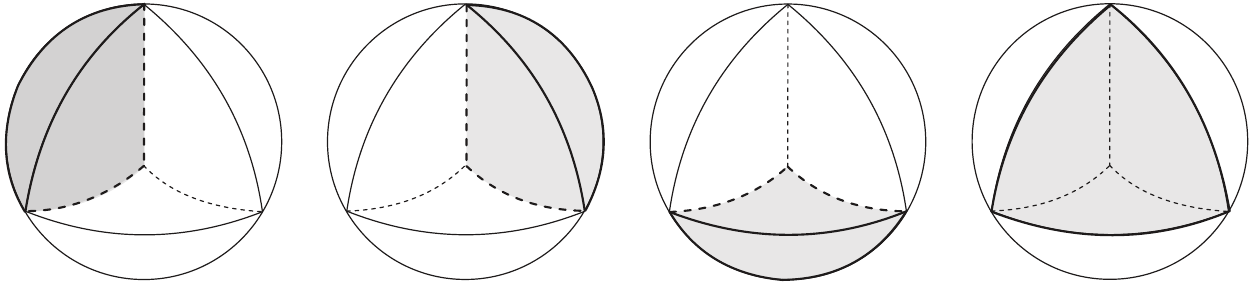}
			\caption{Illustration of a good cover of the sphere.}
			\label{fig:tetraedre}
		\end{figure}
	\end{ex}

	We will now define the bicategory of principal $\cG$-bundles on $X$ in terms of their trivializations relative to a fixed good cover $\bU = \{U_i\}_{i \in I}$. For $n=2,3,...$ we let $ \nonempty{n}\subseteq I^n$ denote the set of tuples $(i_1, \ldots, i_n)$ such that the intersection $U_{i_1\ldots i_n}$ is non-empty. 
	
	\begin{defn}\label{def: Cech data objects}
		A \emph{\v principal $\cG$-bundle on $X$} is a pair $(\rho, \gamma)$, where
		\begin{itemize}
			\item $\rho$ is a collection $\{\rho_{ij} \in  G\}_{ (i,j] \in \nonempty{2}}$ satisfying
			\begin{align}\label{eq: cocycle condition for rho}
				\rho_{ik} = \rho_{ij} \rho_{jk} \text{ for all } (i,j,k) \in \nonempty{3};
			\end{align}
			\item $\gamma$ is a collection $\{\gamma_{ijk} \in A\}_{(i,j,k) \in \nonempty{3}}$ satisfying
			\begin{align}\label{eq: condition for gamma}
				\frac{(\rho_{ij}\cdot \gamma_{jkl})\gamma_{ijl}}{\gamma_{ikl}\gamma_{ijk}} & = \alpha(\rho_{ij}, \rho_{jk}, \rho_{kl}) \text{ for all } (i,j,k,l) \in \nonempty{4};\\
				\gamma_{ijk}& =1_A \text{ whenever } i=j \text{ or } j=k. \label{eq: normalisation condition for gamma}
			\end{align}
		\end{itemize}
	\end{defn}
	\begin{rmk}
		We think of the values $\gamma_{ijk}$ as providing isomorphisms
		\begin{align*}
			[\gamma_{ijk}] \colon \rho_{ik} \to \rho_{ij}\rho_{jk}
		\end{align*}
		in the category $\cG$. Then the condition \eqref{eq: condition for gamma} amounts to the commutativity of the following diagram
		\begin{align*}
			\begin{tikzpicture}[
				xscale=1.8,
				auto,
				every node/.style={font=\small},
				obj/.style={inner sep=3pt} 
				]
				\node[obj] (A1) at (90:2cm)  {$\rho_{ik} \rho_{kl}$};
				\node[obj] (A2) at (18:2cm)  {$(\rho_{ij}\rho_{jk})\rho_{kl}$};
				\node[obj] (A3) at (306:2cm)   {$\rho_{ij}(\rho_{jk}\rho_{kl}).$};
				\node[obj] (A4) at (234:2cm) {$\rho_{ij}\rho_{jl}$};
				\node[obj] (A5) at (162:2cm) {$\rho_{il}$};
				\draw[->] (A1) -- node {$[\gamma_{ijk}] \wedge \id_{\rho_{kl}}$} (A2);
				\draw[->] (A2) -- node {$\alpha(\rho_{ij}, \rho_{jk}, \rho_{kl})$} (A3);
				\draw[->] (A4) -- node {$\id_{\rho_{ij}} \wedge [\gamma_{jkl}]$} (A3);
				\draw[->] (A5) -- node {$[\gamma_{ijl}]$} (A4);
				\draw[->] (A5) -- node {$[\gamma_{ikl}]$} (A1);
			\end{tikzpicture}
		\end{align*}
		In a similar way, the normalization condition arises from requiring $[\gamma_{ijk}]$ to be compatible with the unitors of $\cG$. 
	\end{rmk}
	\begin{rmk}\label{rmk: locally constant}
		Since our cover $\bU$ is good, the collection $\rho$ can be thought of as a locally constant function
		\begin{align*}
			\hat{\rho} \colon \bigsqcup_{i,j \in I} U_{ij} \to G; \qquad &  \hat{\rho}(y) = \rho_{ij} \text{ for } y \in U_{ij}.
		\end{align*}
		Likewise, the data of $\gamma$ is equivalent to a locally constant function $\bigsqcup_{ijk} U_{ijk} \to A$. Analogous remarks apply to the data defining 1-morphisms and 2-mor\-phisms outlined below. 
	\end{rmk}
	
	\begin{defn}\label{def: Cech data one-morphisms}
		Let $(\rho, \gamma), (\rho', \gamma')$ be two principal $\cG$-bundles. A \emph{morphism} $(\rho, \gamma) \to  (\rho', \gamma')$ is a pair $(t,\eta)$, where
		\begin{itemize}
			\item $t$ is a collection $\{t_i \in G \}_{i \in I}$ satisfying
			\begin{align}\label{eq: condition for t}
				t_i \rho_{ij} = \rho'_{ij}t_j \text{ for all } (i,j) \in \nonempty{2};
			\end{align}
			\item $\eta$ is a collection $\{\eta_{ij} \in A \}_{(i,j)\in \nonempty{2}}$ satisfying 
			\begin{align}\label{eq: condition for eta}
				\frac{(\rho'_{ij}\cdot \eta_{jk})\eta_{ij}}{\eta_{ik}}\frac{(t_i \cdot \gamma_{ijk})}{\gamma'_{ijk}} & = \frac{\alpha(t_i, \rho_{ij}, \rho_{jk}) \alpha(\rho'_{ij}, \rho'_{jk}, t_k)}{\alpha(\rho_{ij}, t_j, \rho'_{jk})} 
			\end{align}
		for all $(i,j,k) \in \nonempty{3}$, and
			\begin{align}\label{eq: normalisation condition for eta}
				\eta_{ii}& =1_A \text{ for all } i \in I. 
			\end{align}
		\end{itemize}
	\end{defn}
	
	\begin{defn}\label{def: composition of Cech 1-morphisms}
		Consider three principal $\cG$-bundles with morphisms between them as follows:
		\begin{align*}
			(\rho, \gamma) \xrightarrow{(t, \eta)} (\rho', \gamma') \xrightarrow{(t', \eta')} (\rho'', \gamma''). 
		\end{align*}
		We define the \emph{composition} by $(t', \eta') \circ (t, \eta) = (t't, \theta)$, where $(t't)_i = t'_it_i$, and
		\begin{align}
			\theta_{ij} = (t'_i \cdot \eta_{ij}) \eta'_{ij} \frac{\alpha(\rho''_{ij}, t'_j, t_j) \alpha(t'_i, t_i, \rho_{ij})}{\alpha(t'_i, \rho'_{ij}, t_j)} \text{ for } (i,j) \in \nonempty{2}. 
		\end{align}
	\end{defn}
	
	It is not hard to check that the conditions \eqref{eq: condition for t}, \eqref{eq: condition for eta}, \eqref{eq: normalisation condition for eta} are satisfied, so that $(t't, \theta)$ indeed gives a morphism $(\rho, \gamma) \to (\rho'', \gamma'')$. 
	
	\begin{defn}\label{def: 2-morphisms}
		For two 1-morphisms $(t, \eta), (t, \eta') \colon (\rho, \gamma) \to (\rho', \gamma')$, a \emph{2-morphism} $\omega \colon (t, \eta) \to (t, \eta')$ is given by a collection $\omega =\{\omega_i\in A\}_{i \in I}$ satisfying
		\begin{align}\label{eq: condition for omega}
			\frac{\eta'_{ij}}{\eta_{ij}} = \frac{\rho'_{ij} \cdot \omega_j}{\omega_i} \text{ for all } (i,j) \in \nonempty{2}.
		\end{align}
		Between two 1-morphisms $(t, \eta), (t', \eta')$ with $t \ne t'$, there are no 2-morphisms.   
	\end{defn}
	
	\newcommand{\cv}{\mathbin{\vee}}
	\newcommand{\ch}{\mathbin{\wedge}}
	
	\begin{defn}\label{def: vertical composition}
		Consider three 1-morphisms $(t,\eta), (t,\eta'), (t,\eta''): (\rho,\gamma)\to (\rho',\gamma')$ and two 2-morphisms $\omega:(t,\eta) \to (t,\eta')$ and $\omega':(t,\eta') \to (t,\eta'')$. We define the \emph{vertical composition} by $\omega' \cv \omega = \omega'\omega$, where $(\omega'\omega)_i = \omega'_i\omega_i$.
	\end{defn}
	
	It is easy to check that the condition \eqref{eq: condition for omega} is satisfied, so that $\omega' \cv \omega $ indeed gives a 2-morphism $(t,\eta) \to (t,\eta'')$. 
	
	\begin{defn}\label{def: horizontal composition}
		Consider four 1-morphisms $(t,\eta), (t,\eta'): (\rho,\gamma)\to (\rho',\gamma')$, $(t',\mu),(t',\mu') : (\rho',\gamma') \to (\rho'',\gamma'')$ and two 2-morphisms $\omega:(t,\eta) \to (t,\eta')$ and $\omega':(t',\mu) \to (t',\mu')$. We define the \emph{horizontal composition} by $\omega' \ch \omega =  \omega'(t'\cdot\omega)$, where $\bpr{\omega'(t'\cdot\omega)}_i = \omega'_i( t'_i \cdot \omega_i ) $. 
	\end{defn}
	
	Similary, it is straightforward to check that the condition \eqref{eq: condition for omega} is satisfied, so that $\omega' \ch \omega $ indeed gives a 2-morphism $(t',\mu)\circ (t,\eta) \to (t',\mu')\circ(t,\eta')$. 
	
	\begin{lem}\label{lem: 2-morphisme inverse Bun}
		Let $\omega : (t,\eta) \to (t,\eta')$ a 2-morphism in $\Bun_{\cG}(X)$. There exists a 2-morphism
		\begin{equation*}
			\omega' : (t,\eta') \to (t,\eta)
		\end{equation*}
		which is inverse to $\omega$ with respect to vertical composition, i.e. $ \omega' \cv \omega = \id_{(t,\eta)}$ and $\omega \cv \omega' = \id_{(t,\eta')}$.
	\end{lem}
	
	\begin{proof}
		Let $\omega : (t,\eta) \to (t,\eta')$ be a 2-morphism between $(t,\eta), (t,\eta') : (\rho,\gamma) \to (\rho',\gamma')$. In particular, we have
		\begin{align*}
			\eta'_{ij}\omega_i = (\rho'_{ij}\cdot \omega_j)\eta_{ij}.
		\end{align*}
		Define $\omega'$ by $\omega'_i := \omega_i^{-1}$.
		Then $\omega' : (t,\eta') \to (t,\eta)$ is a 2-morphism if and only if
		\[
		\eta_{ij}\omega'_i = (\rho'_{ij}\cdot \omega'_j)\eta'_{ij}.
		\]
		But we have
		\begin{equation*}
			\frac{\eta_{ij} \omega'_i}{(\rho'_{ij}\cdot \omega'_j) \eta'_{ij}} = \frac{\eta_{ij} \omega_i^{-1}}{(\rho'_{ij}\cdot \omega_j^{-1}) \eta'_{ij}} = \frac{\eta_{ij} (\rho'_{ij}\cdot \omega_j)}{\eta'_{ij} \omega_i}.
		\end{equation*}
		This equals $1_A$, since $\omega$ is a 2-morphism and $A$ is abelian. Hence $\omega'$ is indeed a 2-morphism.
		
		It remains to check that $\omega'$ is the vertical inverse of $\omega$, but this is clear from the definition of vertical composition. This proves that every 2-mor\-phism is in fact a 2-isomorphism.
	\end{proof}
	
	\begin{lem}\label{def: 1-morphism inverse}
		Let $(t,\eta): (\rho,\gamma)\to (\rho',\gamma')$ a 1-morphism. This 1-morphism admits a left inverse $(t^{-1},\tilde{\eta})$ satisfying $(t^{-1},\tilde{\eta})\circ (t,\eta) = \id_{(\rho,\gamma)} = (1_G,1_A)$, where
		\begin{equation}
			\tilde{\eta}_{ij} = \frac{\alpha(t_i^{-1}, \rho'_{ij},t_j)}
			{(t_i^{-1}\cdot\eta_{ij})\alpha(\rho_{ij}, t_j^{-1}, t_j) \alpha(t_i^{-1},t_i, \rho_{ij})}.
		\end{equation}
		Moreover, $(t,\eta)\circ (t^{-1},\tilde{\eta})$ is isomorphic to $\id_{(\rho',\gamma')}$, but is not equal to it.
	\end{lem}
	\begin{proof}
		It is not hard to check that $(t^{-1},\tilde{\eta})$ is a 1-morphism: it suffices to verify equations \eqref{eq: condition for t}, \eqref{eq: condition for eta}, and \eqref{eq: normalisation condition for eta}, using that $\alpha$ is a normalized 3-cocycle, and that $(t, \eta)$ is itself a 1-morphism.
		
		Moreover, by the definition of the composition of 1-morphisms, we have $(t^{-1},\tilde{\eta})\circ (t,\eta) = (1_G,\theta) = \id_{(\rho,\gamma)}.$
		On the other hand, we can define a 2-morphism $\omega \colon (t,\eta)\circ \pr{t^{-1},\tilde{\eta}} \to \id_{(\rho',\gamma')}$ by setting $\omega_i=\alpha(t_i,t_i^{-1},t_i)$. It is straightforward to verify that $\omega$ defines a 2-morphism. Moreover, it follows from Lemma \ref{lem: 2-morphisme inverse Bun} that $\omega$ is an isomorphism.
	\end{proof}
	
	\begin{lem}
		Every 1-morphism of $\Bun_\cG(X)$ is in fact a 1-isomorphism.
	\end{lem}
	\begin{proof}
		This follows directly from Lemma \ref{def: 1-morphism inverse}.
	\end{proof}

	\begin{lem}
		The composition of three 1-morphisms is weakly but not strict\-ly associative. 
	\end{lem}
	
	\begin{proof}
		Suppose we have three 1-morphisms, together with their associated compositions, as shown in the following diagram:
		\begin{equation*}
			\begin{tikzcd}
				{(\rho,\gamma)} && {(\rho',\gamma')} && {(\rho'',\gamma'')} && {(\rho''',\gamma''')}.
				\arrow["{(t,\eta)}", from=1-1, to=1-3]
				\arrow["{{(t't,\theta)}}"', bend right=25, from=1-1, to=1-5]
				\arrow[""{name=0, anchor=center, inner sep=0}, "{{(t''(t't),\theta')}}"',bend right=40 , from=1-1, to=1-7]
				\arrow[""{name=1, anchor=center, inner sep=0}, "{{((t''t')t,\vartheta')}}",bend left=40, from=1-1, to=1-7]
				\arrow["{{(t',\eta')}}", from=1-3, to=1-5]
				\arrow["{{(t''t',\vartheta)}}", bend left=25, from=1-3, to=1-7]
				\arrow["{{(t'',\eta'')}}", from=1-5, to=1-7]
			\end{tikzcd}
		\end{equation*} 
		If composition were strictly associative, we would have $(t''t't,\theta')=(t''t't,\vartheta')$. However, a direct computation instead yields that 
		\begin{align*}
			\theta'_{ij} = \frac{(t''_it'_i \cdot \eta_{ij})(t''_i\cdot \eta'_{ij})   \eta''_{ij} \Bigl( t''_i\cdot \bigl( \alpha(\rho''_{ij}, t'_j,t_j)   \alpha(t'_i,t_i,\rho_{ij})\bigr) \Bigr) \alpha(\rho'''_{ij}, t''_j,t'_jt_j)   \alpha(t''_i,t'_it_i,\rho_{ij}) }{\bigl(t''_i\cdot \alpha(t'_i,\rho'_{ij}, t_j)\bigr) \alpha(t''_i,\rho''_{ij}, t'_j t_j)}
		\end{align*}
		and 
		\begin{align*}
			\vartheta'_{ij}= \frac{(t''_it'_i \cdot \eta_{ij}) (t''_i \cdot \eta'_{ij})  \eta''_{ij} \alpha(\rho'''_{ij}, t''_j,t'_j) \alpha(t''_i,t'_i,\rho'_{ij})  \alpha(\rho'''_{ij}, t''_jt'_j,t_j)   \alpha(t''_it'_i,t_i,\rho_{ij})}{\alpha(t''_i,\rho''_{ij}, t'_j) \alpha(t''_it'_i,\rho'_{ij}, t_j)}.
		\end{align*}
		Using the properties of $\alpha$ together with equation \eqref{eq: condition for t}, we obtain
		\begin{align*}
			\theta'_{ij} \alpha(t''_i,t'_i,t_i) = \bigl( \rho'''_{ij}\cdot \alpha(t''_j,t'_j,t_j)\bigr)\vartheta'_{ij}.
		\end{align*}
		So, the composition is not strictly associative. Instead, setting $\tilde{\alpha}_i =\alpha(t''_i,t'_i,t_i)$, we obtain that $\tilde{\alpha}$ is a 2-morphism from $(t''t't,\theta')$ to $(t''t't,\vartheta')$. 
	\end{proof}
	
	\begin{prop}
		The above data form a bicategory $\Bun_\cG(X)$ with objects principal $\cG$-bundles. 
	\end{prop}
	\begin{proof}
		We define the \emph{identity 1-morphism} $\id_{(\rho,\gamma)}$ to be $(1_G,1_A)$, where $t_i=1_G$ for every $i\in I$ and $\eta_{ij}=1_A$ for every $i,j\in I$. The \emph{left and right unitors} are given by the identity 2-morphism, with $\omega_i=1_A$ for every $i\in I$. It is then straightforward to verify that all the remaining bicategory axioms are satisfied by $\Bun_\cG(X)$.
	\end{proof}
	
	\begin{rmk}\label{rmk: compare BCMNP}
		Let us compare this definition with the bicategory of \cite[Definition 2.28]{BCM} or equivalently \cite[Notation 4.12]{BCMNP}. The first difference is that the authors of \cite{BCM, BCMNP} work with arbitrary coverings, not necessarily good covers, so they need to use the language of locally constant functions as in Remark \ref{rmk: locally constant} rather than collections of elements in $G$ and $A$. 
		
		The authors of \cite{BCM, BCMNP} also allow multiple different coverings, rather than fixing a single good cover as we have done here. This results in slightly more complicated definitions of morphisms, 2-morphisms, and compositions. We will see below that this is not necessary for our purposes, as any principal $\cG$-bundle trivializes over a good cover and hence comes from \v Cech data on such a cover. It is sometimes necessary to work with more general covers for several reasons: in \cite{BCMNP}, the authors wish to work over stacks, rather than manifolds (for example to be able to define equivariant $\cG$-bundles), so that a good cover may not exist. Even when working over a manifold, in \cite{BCM} the authors often wish to use \v Cech data defined relative to the universal cover of the base, which is also not a good cover. These motivations do not apply to our current work, so that we can simplify our bicategory by fixing a good cover. 
		
		A final difference is that we allow non-trivial actions of $G$ on $A$; reducing to the case of a trivial action, we recover the same conditions as \cite{BCM, BCMNP}. For example, our e
		equations \eqref{eq: condition for gamma}, \eqref{eq: condition for eta}, and \eqref{eq: condition for omega} simplify to equations (2.20), (2.22), and (2.23) of \cite{BCM}, respectively. 
	\end{rmk}
	
	\begin{rmk} Let us consider the bicategory $\Bun_{*\sq A}(X)$. We see that objects correspond to \v{C}ech 2-cocyles $\hat{\gamma} \colon  \sqcup_{i,j,k \in I}U_{ijk} \to A$, 1-morphisms are 1-cochains $\hat{\eta} \colon \sqcup_{i,j \in I} U_{ij} \to A$ such that $d\hat{\eta} = \hat{\gamma}/\hat{\gamma}'$, and 2-morphisms are 0-cochains with an analogous property. In other words, we recover the bicategory of $A$-gerbes on $X$, as expected. 
	\end{rmk}
	
	\section{\texorpdfstring{Principal $\cG$-bundles on $S^2$}{Principal G-bundles on S2}}\label{sec: Sandrine}

			In this section, we will analyze $\Bun_\cG(S^2)$ in detail, for $\bU$ the good cover of Example \ref{ex: cover}. The goal is to produce a skeletal presentation of the bicategory, so that we can compare it easily to the bicategory of holonomy data. 
			
			\subsection{\texorpdfstring{Principal $\cG$-bundles on $S^2$: objects}{Principal bundles on S2:objects}}
			
			\begin{prop}\label{prop: obj of Bung}
				There is a bijection 
				\begin{align*}
					Obj\bpr{\Bun_\cG(S^2)} \EquivTo G^3\times A^{10}.
				\end{align*}
			\end{prop}
			
			\begin{proof}
				Let $(\rho, \gamma)$ be a principal $\cG$-bundle on $S^2$. We know that $\rho_{ij}$ can take up to 16 values of $G$ because $U_{ij}$ are non-empty for all $i, j \in \{1,2,3,4\}$.  However, if we fix $\rho_{12},\rho_{23}$ and $\rho_{34}$, we can determine all other $\rho_{ij}$ using equation \eqref{eq: cocycle condition for rho}. We will therefore call $\rho_{12}, \rho_{23}, \rho_{34}$ ``generators'' of $\rho$. 
				
				Similary, $\gamma_{ijk}$ can take up to 64 values of $A$ (one for each $U_{ijk}$). But if $i=j$ or $j=k$, equation \eqref{eq: normalisation condition for gamma} implies that $\gam{ijk}=1_A$. 
				
				If $i=k$ and $i\neq j$, then equation \eqref{eq: condition for gamma}, evaluated at $(i,j,i,j)$, gives a relation between $\gam{iji}$ and $\gam{jij}$, and this relation is equivalent to equation \eqref{eq: condition for gamma} evaluated at $(j,i,j,i)$. So, if we fix $\gamma_{121},\gamma_{131}, \gamma_{141}, \gamma_{232}, \gamma_{242}, \gamma_{343}$, we can determine all other $\gamma_{iji}$ for $i\neq j$.

				It remains to consider the case where $i, j, k$ are distinct. Suppose we fix $\gamma_{ijk}$ for $1 \le i < j < k \le 4$ (that is, $\gamma_{123}$, $\gamma_{124}$, $\gamma_{134}$, and $\gamma_{234}$). To obtain the other values of $\gamma$ in terms of these $\gamma_{ijk}$, it suffices to study the permutations of $(i,j,k)$. In particular, we know that the set of permutations of three elements $S_3 \cong D_3$ is generated by a rotation (a 3-cycle), denoted by $P$, and a transposition, denoted by $T$, such that $T^2 = \id$, $P^3 = \id$, and $(PT)^2 = \id$.
				
				We would like to define the rotation map (3-cycle) $P$ by
				\begin{equation*}
					P(\gamma_{ijk}) = \rho_{ji} \cdot \left( \frac{\gamma_{ijk}\gamma_{iki}\alpha_{ijki}}{\gamma_{iji}} \right) = \gamma_{jki},
				\end{equation*}
				using equation \eqref{eq: condition for gamma} evaluated at $(i,j,k,i)$.
				From $P$, we can therefore obtain $\gamma_{jki} = P(\gamma_{ijk})$ in terms of $\gamma_{ijk}$, as well as $\gamma_{kij} = P^2(\gamma_{ijk})$, also expressed in terms of $\gamma_{ijk}$.
				
				Moreover, we define the transposition map $T$ by
				\begin{equation*}
					T(\gamma_{ijk}) = \rho_{ji} \cdot \left( \frac{\gamma_{iji}\alpha_{ijik}}{\gamma_{ijk}} \right) = \gamma_{jik},
				\end{equation*}
				using equation \eqref{eq: condition for gamma} evaluated at $(i,j,i,k)$. Thus we can determine $\gamma_{jik} = T(\gamma_{ijk})$, $\gamma_{ikj} = P \circ T(\gamma_{ijk})$, and $\gamma_{kji} = T \circ P(\gamma_{ijk})$.
				
				To ensure that these are reasonable definitions, we need to check the relations $T^2 = \mathrm{id}$, $P^3 = \mathrm{id}$, and $(PT)^2 = \mathrm{id}$ hold. This is a straightforward calculation, using the fact that $\alpha$ is a normalized 3-cocycle. It follows that for a fixed triple $(i<j<k)$, the single value $\gamma_{ijk}$ determines the value of $\gamma$ on all permutations of this triple.
				
				Finally, recall that the quadruple intersection is empty if the opens are distinct. This implies that equation \eqref{eq: condition for gamma} is not required to hold for $(i,j,k,l)$ distinct. In other words, there is no relation between $\gamma_{123}$, $\gamma_{124}$, $\gamma_{134}$, and $\gamma_{234}$, since $\bigcap_{i=1}^4 U_i = \varnothing$.

				In summary, if we fix $\gamma_{121},\gamma_{131}, \gamma_{141}, \gamma_{232}, \gamma_{242}, \gamma_{343}, \gamma_{123}, \gamma_{124}, \gamma_{134}$ and $\gamma_{234}$ (the ``generators'' of $\gamma$), we can determine all other $\gamma_{ijk}$ using equations \eqref{eq: condition for gamma} and \eqref{eq: normalisation condition for gamma}.
				
				A object $(\rho,\gamma)$ of $\Bun_\cG(S^2)$ is therefore determined by three free values of $G$ and ten free values of $A$, as claimed. 
			\end{proof}
			
			\begin{rmk}
				The bijection of Proposition \ref{prop: obj of Bung} is determined by our choice of generators. We could have made other choices (for example $\rho_{43}, \rho_{32}, \rho_{21}$), and would have obtained a different form for the bijection. We will similarly obtain non-canonical bijections in our descriptions of 1-morphisms and 2-morphisms below. However, we fix once and for all our choice of generators in each of these cases, so that our bijections are fixed throughout this paper. 
			\end{rmk}
			\begin{rmk}
				See Tables \ref{tab: rho_ij possible} and \ref{tab: gamma possibles} in Appendix \ref{sec: appendix table} for a complete description of each component of $\rho, \gamma$ in terms of the chosen generators.
			\end{rmk}
			
			\subsection{\texorpdfstring{Principal $\cG$-bundles on $S^2$: 1-morphisms}{Principal bundles on S2: 1-morphisms}}
			
			\begin{prop}
				Let $(\rho,\gamma),(\rho',\gamma') \in Obj\bpr{\Bun_\cG^\bU(S^2)}$ be two objects, determined by their generators under the bijection of Proposition \ref{prop: obj of Bung}. Then there is a bijection   
				\begin{align*}
					\Hom_{\Bun_\cG^\bU(S^2)} \bpr{(\rho,\gamma),(\rho',\gamma')} & \EquivTo G^* \times A^3\\
					(t,\eta) &\mapsto \bpr{t_1, (\eta_{12},\eta_{23},\eta_{34})}
				\end{align*}
				where $G^*$ is the subset of elements $t_1\in G$ satisfying the \emph{concordance condition} 
				\begin{equation}\label{eq: concordance condition}
					\frac{ {\bigl( (t_1 \rho_{12}) \cdot \gamma_{234}\bigl)} \pr{t_1\cdot \gamma_{124}}\gamma'_{123}\gamma'_{134}  \alpha(\rho'_{12},\rho'_{23},\rho'_{34})}{\pr{\rho_{12}'\cdot \gamma'_{234}}\gamma'_{124}  {\bigl(t_1\cdot (\gamma_{123}\gamma_{134})\bigl) } \pr{t_1 \cdot\alpha(\rho_{12},\rho_{23},\rho_{34})}} =1_A. 
				\end{equation}
			\end{prop}
			
			\begin{rmk}
				We emphasize that the definition of the set $G^*$ depends on the pair of objects under consideration. 
			\end{rmk}
			
			\begin{proof}
				Let $(t,\eta): (\rho,\gamma)\to (\rho',\gamma')$. If we fix $t_1$, we can determine all other $t_i$ using equation \eqref{eq: condition for t}. We will thus call $t_1$ the ``generator'' of $t$. 
				
				Also, if we fix $\eta_{12},\eta_{23}$ and $\eta_{34}$ (the ``generators'' of $\eta$), we can determine all other $\eta_{ij}$ using equation \eqref{eq: condition for eta}. Indeed, equation \eqref{eq: condition for eta} evaluated at $(i,j,i)$ for $i \neq j$, gives a relation between $\eta_{ij}$ and $\eta_{ji}$, which can be expressed using the rule
				\begin{equation*}
					S(\eta_{ij})=\rho'_{ji}\cdot \left(\frac{\gamma'_{iji} \alpha(t_i,\rho_{ij},\rho_{ji}) \alpha(\rho'_{ij},\rho'_{ji},t_i)} {\eta_{ij} \bigl(t_i \cdot \gamma_{iji}\bigr) \alpha(\rho'_{ij},t_j,\rho_{ji})} \right) =\eta_{ji}.
				\end{equation*}
				Naturally, one would therefore want $S\bigl(S(\eta_{ij})\bigr)=S(\eta_{ji})=\eta_{ij}$. This follows easily by applying the fact that $\alpha$ is a normalized 3-cocycle four times.
				
				Moreover, if $i,j,k$ are distinct, then equation \eqref{eq: condition for eta} gives a relation between $\eta_{ij}$, $\eta_{jk}$ and $\eta_{ik}$, which can be expressed using the rule
				\begin{align*}
					C(\eta_{ij},\eta_{jk}) = \frac{ \eta_{ij} \bigl(\rho'_{ij}\cdot \eta_{jk}\bigr) \bigl(t_i\cdot \gamma_{ijk}\bigr) \alpha(\rho'_{ij},t_j,\rho_{jk}) }{ \gamma'_{ijk} \alpha(t_i,\rho_{ij},\rho_{jk}) \alpha(\rho'_{ij},\rho'_{jk},t_k)} = \eta_{ik}. 
				\end{align*}
				
				We note that an element $\eta_{ki}$ can be expressed in two different ways in terms of $\eta_{ij}$ and $\eta_{jk}$:
				\begin{equation*}
					\eta_{ki}=S(\eta_{ik})=S\bigl(C(\eta_{ij},\eta_{jk})\bigr) \quad \text{or} \quad
					\eta_{ki}=C(\eta_{kj},\eta_{ji})=C\bigl(S(\eta_{jk}),S(\eta_{ij})\bigr).
				\end{equation*}
				Thus, we must show that these two approaches coincide for all distinct $i,j,k \in I$, which amounts to checking that
				\begin{equation*}
					S\bigl(C(\eta_{ij},\eta_{jk})\bigr) =C\bigl(S(\eta_{jk}),S(\eta_{ij})\bigr).
				\end{equation*}
				This is tedious but not hard: we use the fact that $\alpha$ is a normalized 3-cocycle nine times.
				
				Finally, for the quadruple $(i,j,k,l)$, the rule $C$ gives two different formulas for $\eta_{il}$:
				\begin{equation*}
					C\bpr{\eta_{ij},C(\eta_{jk}, \eta_{kl})}=  C\bpr{C(\eta_{ij},\eta_{jk}), \eta_{kl}}
				\end{equation*}
				Using the definition of $C$ and the fact that $\alpha$ is a 3-cocycle, we see that the previous equation is  equivalent to 
				\begin{equation}\label{eq : condition de concordance ijkl}
					\frac{{\bigl( (t_i \rho_{ij}) \cdot \gam{jkl}\bigl)} \pr{t_i\cdot \gamma_{ijl}}\gamma'_{ijk}\gamma'_{ikl}  \alpha(\rho'_{ij},\rho'_{jk},\rho'_{kl})}{\pr{\rho_{ij}'\cdot \gamma'_{jkl}}\gamma'_{ijl}  {\bigl(t_i\cdot (\gamma_{ijk}\gamma_{ikl})\bigl) } {\bigl(t_i \cdot\alpha(\rho_{ij},\rho_{jk},\rho_{kl})\bigl)}} = 1_A.
				\end{equation}
				At first glance, this gives $24$ equations of the form \eqref{eq : condition de concordance ijkl}, one for each permutation of $(i,j,k,l)$. However, one can show that all these equations are actually equivalent to one another by using the fact that $\alpha$ is a normalized 3-cocycle and equation \eqref{eq: condition for gamma}. Choosing $(i,j,k,l)=(1,2,3,4)$ yields the concordance condition \eqref{eq: concordance condition} in terms of our generators.
			\end{proof}
			
			\begin{rmk}
				See Tables \ref{tab: t_i} and \ref{tab: eta_ij possible} in Appendix \ref{sec: appendix table} for the values of all components of $t$ and $\eta$ in terms of the fixed generators. 
			\end{rmk}
			
			\begin{cor}\label{cor: 1-morphism of BunG}
				Let $g\in G$ and $(\rho,\gamma)$ be an object. Then $(\rho,\gamma)$ is 1-isomorphic  to an object $(\rho',\gamma')$ with : 
				\begin{itemize}
					\item all generators of $\rho'$ equal to $1_G$ ;
					\item all generators of $\gamma'$ equal to $1_A$, except for
					\begin{equation*}
						\gamma'_{123} = g \cdot \pr{\frac{\gamma_{123}\gamma_{134}   \alpha(\rho_{12},\rho_{23},\rho_{34})} {\pr{\rho_{12} \cdot \gamma_{234}} \gamma_{124}}}.
					\end{equation*}
				\end{itemize}
				We denote objects of this form by $(1_G,\gamma_{123}')$.
			\end{cor}
			
			\begin{proof}
				Take the 1-morphism $(t,\eta)$ generated by $t_1=g$ and arbitrary $\eta_{12}, \eta_{23},$ $\eta_{34}$ in $A$. 
			\end{proof}
			
			\begin{cor} \label{rmk : objets Bun}
				Two objects of the form $(1_G,\gamma_{123})$ and $(1_G,\gamma_{123}')$ as in Corollary \ref{cor: 1-morphism of BunG} are isomorphic if and only if there exists $g\in G$ such that $\gamma'_{123} = g\cdot \gamma_{123}$. 
			\end{cor}
			\begin{proof}
				It suffices to replace $(\rho,\gamma)$ with $(1_G,\gamma_{123})$ in Corollary \ref{cor: 1-morphism of BunG}.
			\end{proof}
			
			\subsection{\texorpdfstring{Principal $\cG$-bundles on $S^2$: 2-morphisms}{Principal G-bundles on S2: 2-morphisms}}
			\begin{prop}
				A 2-morphism $\omega$ is determined by one free value of $A$.
			\end{prop}
			
			\begin{proof}
				Let $\omega: (t,\eta)\to (t,\eta')$. If we fix $\omega_1$, we can determine all other $\omega_i$ using equation \eqref{eq: condition for omega}. 
			\end{proof}
			
			\begin{rmk}
				We therefore call $\omega_1$ the generator of $\omega$. See Table \ref{tab: omega possibles} in Appendix \ref{sec: appendix table} for the values of the other components of $\omega$ in terms of the generator $\omega_1$. 
			\end{rmk}
			
			\begin{cor}\label{cor: cech 2-morphisms A}
				A 1-morphism $(t,\eta)$ is 2-isomorphic to $(t',\eta')$ if and only if $t=t'$. There is a bijection between $2\Hom((t,\eta),(t,\eta'))$ and $A$.
			\end{cor}
			
			\begin{proof}
				By Definition \ref{def: 2-morphisms}, there are no 2-morphisms unless $t=t'$. On the other hand, if $t=t'$, any $m \in A$ gives rise to a 2-morphism $\omega$  between $(t,\eta)$ and $(t,\eta')$ generated by $\omega_1=m$. 
			\end{proof}

			\subsection{\texorpdfstring{Skeleton of $\Bun_\cG(S^2)$}{Skeleton of BunG}}
			We are now in a position to identify a skeletal subcategory of $\Bun_\cG(S^2)$. 
			
			\begin{defn}\label{def: S}
				Choose a set $C\subset A$ of representatives for the orbits of the action of $G$ on $A$. Define $\cS$ to be the bicategory with
				\begin{itemize}
					\item objects $a \in C$;
					\item 1-morphisms $\Hom_\cS(a,a') = \left\{ \begin{array}{cc}
						\Stab_G(a) & \text{ if } a = a'  \\
						\varnothing & \text{ else;}
					\end{array} \right.$
					\item 2-morphisms $2\Hom_\cS(g,g') = \left\{ \begin{array}{cc}
						A &  \text {if }g = g'\\
						\varnothing & \text{else;}
					\end{array}
					\right.$
					\item composition of 1-morphisms given by multiplication in $G$;
					\item associator given by $\alpha$;
					\item identity 1-morphisms given by $\id_{a} = 1_G$;
					\item left and right unitors given by $1_A$;
					\item vertical composition of 2-morphisms given by multiplication in $A$;
					\item identity 2-morphism given by $\id_{g} = 1_A$;
					\item horizontal composition of 2-morphisms given by the diagram 
					\begin{align*}
						\begin{tikzpicture}[scale=.8, every node/.style={font=\scriptsize}]
							\node (A) at (0,0) {$(a,a')$};
							\node (B) at (3,0) {$(b,g \cdot a')$};
							\node (C) at (6.3,0) {$(c,g'g \cdot a')$};
							\node (D) at (8.2,0) {$(a,a')$};
							\node (E) at (13, 0) {$(c,g'g \cdot a')$};
							\draw[->] (A) to[bend left=40] node[name=f, above] {$g$} (B);
							\draw[->] (A) to[bend right=40] node[name=fprime, below] {$g$} (B);
							\draw[->] (B) to[bend left=40] node[name=g, above] {$g'$} (C);
							\draw[->] (B) to[bend right=40] node[name=gprime, below] {$g'$} (C);
							\draw[->] (D) to[bend left=35] node[name=h, above] {$g'g$} (E);
							\draw[->] (D) to [bend right=35] node[name=hprime, below] {$g'g$} (E);
							\draw[->, double, double distance=1.2pt, shorten <=5pt, shorten >=5pt] (f) -- node[right] {$m$} (fprime);
							\draw[->, double, double distance=1.2pt, shorten <=5pt, shorten >=5pt] (g) -- node[right] {$m'$} (gprime);
							\draw[->, double, double distance=1.2pt, shorten <=5pt, shorten >=5pt] (h) -- node[right] {$m' \wedge m$} (hprime);
						\end{tikzpicture}
					\end{align*}
					with $m' \wedge m= m' (g' \cdot m) \in A=2\Hom_\cS(g'g, g'g).$
				\end{itemize}
			\end{defn}
			
			\begin{prop}
				This data defines a skeletal bicategory $\cS$. 
			\end{prop}
			
			\begin{proof}
				It is straightforward to check that $\cS$ satisfies each of the conditions in the definition of a bicategory. 
				It is skeletal by construction: two objects of $\cS$ are 1-isomorphic if and only if they are equal, and two 1-morphisms of $\cS$ are 2-isomorphic if and only if they are equal.
			\end{proof}
			
			We also note the following immediate consequence of the definition, which will be useful later on.
			
			\begin{lem}\label{lemma: S automorphism 2-group}
				Let $a \in C$ be an object of $\mathcal{S}$. Let $\alpha_|$ be the restriction of the subgroup of the cocycle $\alpha$ to the subgroup $\Stab_G(a) \subset G$. Then the automorphism 2-group of $C$ is equal to $\cG(\Stab_G(a), A, \alpha_\vert)$.
			\end{lem}

			\begin{thm} \label{thm: S and Bun}
				We have an equivalence of bicategories
				$$\cS \cong \Bun_\cG(S^2).$$
			\end{thm}
			
			\begin{proof}
				We define a strict 2-functor $\Psi : \cS \to \Bun_\cG(S^2)$ by: 
				\begin{itemize}
					\item objects : $\Psi(a) = (\rho, \gamma) :=(1_G,a) $ with $\rho_{ij} = 1_G$ for all $i,j\in I$ and $\gamma_{ijk} = 1_A$ for all generators except $\gamma_{123} =a$;
					\item 1-morphisms : $\Psi(g)=(t,\eta)  := (g,1_A)$ with $t_1=g$ and $\eta_{ij} = 1_A$ for all generators;
					\item 2-morphisms : $\Psi(b)= \omega :=b $ with $\omega_1=b$.  
				\end{itemize}
				That $\Psi$ is a strict 2-functor follows naturally from our construction. To show that $\Psi$ is an equivalence of bicategories, we will prove that it is essentially surjective (in Lemma \ref{lem: Psi essentially surjective}) and that it induces equivalences of hom-categories (in Lemma \ref{lem: Psi equivalence functor}). 
			\end{proof}

			\begin{lem}\label{lem: Psi essentially surjective}
				The functor $\Psi$ is essentially surjective. 
			\end{lem}
			
			\begin{proof}
				Let $(\rho,\gamma)$ be an object of $\Bun_\cG(S^2)$. We need to find an object $a \in \cS$ such that $(\rho,\gamma) \cong \Psi(a)$. We consider the element
				\begin{equation*}
					{\frac{\gamma_{123}\gamma_{134}   \alpha(\rho_{12},\rho_{23},\rho_{34})} {\pr{\rho_{12} \cdot \gamma_{234}} \gamma_{124}}} \in A.
				\end{equation*}
				By definition of $C$, there exists an element $a \in C$ in the same $G$-orbit: that is, there exists some $g \in G$ such that 
				\begin{equation*}
					a =g \cdot \pr{\frac{\gamma_{123}\gamma_{134}   \alpha(\rho_{12},\rho_{23},\rho_{34})} {\pr{\rho_{12} \cdot \gamma_{234}} \gamma_{124}}}.
				\end{equation*}
				By Corollary \ref{cor: 1-morphism of BunG}, $g$ determines a 1-isomorphism between $(\rho,\gamma)$ and $\Psi(a)$.
			\end{proof}
			
			\begin{lem}\label{lem: Psi equivalence functor}
				Let $a,a'\in \cS$. The functor 
				\begin{align*}
					\Psi_{a,a'}: \Hom_{\cS} \pr{a,a'} \to \Hom_{\Bun_\cG^Y (S^2)}(\Psi(a),\Psi(a'))
				\end{align*}
				is an equivalence of categories. 
			\end{lem}
			\begin{proof}
				If $a\neq a'$, $\Hom_{\cS} \pr{a,a'} = \varnothing$. On the other hand, by Corollary \ref{rmk : objets Bun}, $\Hom_{\Bun_\cG^Y (S^2)}(\Psi(a),\Psi(a'))\neq \varnothing$ is non-empty if and only if $a$ and $a'$ are in the same $G$-orbit; but since we chose $C$ to consist of representatives of each orbit, this occurs if and only if $a=a'$. So we assume for the rest of the proof that $a=a'$ and we denote $\Psi_{a,a'} = \Psi_a$. We need to show that $\Psi_a$ is fully faithful and essentially surjective.
				
				To show that $\Psi_a$ is fully faithful, we take $g,h \in \Hom_{\cS} \pr{a,a}$. We need to show that the induced map
				\begin{align*}
					\Psi_{g,h}: 2\Hom_{\cS} \pr{g,h} &\to 2\Hom_{\Bun_\cG^Y (S^2)}\bpr{\Psi(g),\Psi(h)}\\
					b & \mapsto \Psi(b).
				\end{align*} 
				is a bijection. We note that if $g \ne h$ then both 
				$2\Hom_{\Bun_\cG^Y (S^2)}\pr{\Psi(g),\Psi(h)}$ and $2\Hom_{\cS} \pr{g,h}$ are empty. Therefore, we consider the case where $g=h$ and $2\Hom_{\cS} \pr{g,g} =A$. Then $\Psi_{g,g}$ is just the bijection of Corollary \ref{cor: cech 2-morphisms A}. 
				
				To show that $\Psi_a$ is essentially surjective, we take $(t,\eta)\colon \Psi(a) \to \Psi(a)$ in $\Bun_\cG^Y(S^2)$. We want to find some $g \in \Hom_{\cS} \pr{a,a}$ such that $\Psi_a(g) \cong (t,\eta)$.
				Recall that $\Psi(a)=(1_A,a)$ and that $(t,\eta)$ is a 1-morphism. Thus, it satisfies the concordance condition \eqref{eq: concordance condition}, which in this case reduces to the equation
				\[
				\frac{a}{t_1\cdot a}=1_A.
				\]
				This implies that $t_1 \in \Stab_G(a)$, and so we take $g:=t_1$. We observe that $\Psi(g)$ and $(t,\eta)$ have the same $t$-component and thus by Corollary \ref{cor: 1-morphism of BunG} they are 2-isomorphic.
				
				We have thus shown that $\Psi_{a,b}$ is an equivalence.
			\end{proof}
			
			\section{Holonomy data}\label{sec: Yves}
			
			\begin{defn}\label{def:BG for 2-group}
				Let $\cG$ be a 2-group (or in fact a monoidal category, more generally). We denote by $* \sq \cG$ the bicategory with a single object $*$, 1-morphisms given by the objects of $\cG$, and 2-morphisms given by the morphisms of $\cG$. The composition of 1-morphisms and the horizontal composition of 2-morphisms come from the multiplication functor on $\cG$, whereas the vertical composition of 2-morphisms is just the ordinary composition of morphisms in $\cG$. 
			\end{defn}
			
			\begin{defn}\label{defn: holonomy}
				Let $X$ be a smooth connected manifold and $\cG$ a 2-group. The bicategory of \emph{$\cG$-holonomy data} on $X$ is the bicategory of functors
				\begin{align*}
					\Fun(*\sq \pi_{\le 2}(X), * \sq \cG).
				\end{align*}
			\end{defn}
			
			Our goal is therefore to analyze the bicategory $\Fun(*\sq \pitwoS, * \sq \cG)$ for $\cG= \cG(G, A, \alpha)$ as in Example \ref{ex: G, A, alpha}. Combining Definitions \ref{defn: fundamental 2-group} and \ref{def:BG for 2-group}, we see that $*\sq \pitwoS$ is the strict 2-category with one object $*$, one 1-morphism $\id_*$, and 2-morphisms $2\Aut_{* \sq \pitwoS}(\id_*) = \mathbb{Z}$ (with composition given by addition). 
			
			\subsection{\texorpdfstring{Holonomy on $S^2$: objects}{Holonomy on s2 : objets}}
			\begin{prop}\label{prop: holonomy objects}
				There is a bijection 
				\begin{align*}
					\text{Obj}(\Fun(* \sq \pitwoS, * \sq \cG)) & \EquivTo A^2\\
					F & \mapsto (\fcomp, f(1)).
				\end{align*}
			\end{prop}
			
			\begin{proof}
				Let us consider an object $F \in \Fun(*\sq \pitwoS, * \sq \cG)$, and analyze how it acts as a functor.
				
				On the level of objects, we have no choice but to set $F(*)$ to be the single object of $* \sq \cG$, which we will denote by  $\ptG $.
				
				On the level of 1-morphisms, we know that $F(\id_*)$ must be an object of $\cG$ which is isomorphic to the monoidal unit $1_G$. Since $\cG$ is skeletal, we have no choice but to set $F(\id_*) = 1_G$. However, we have a choice of 2-morphism $F_0 \colon 1_G \to F(\id_*) $ given by $F_0 = [f_0]_{1_G}$ for an element $f_0 \in A$, as well as a choice of 2-morphism $F_{\id_*, \id_*} \colon F(\id_*) \circ F(\id_*) \to F(\id_* \circ \id_*)$ given by $F_{\id_*, \id_*} = [\fcomp]_{1_G}$ for an element $\fcomp \in A$. The compatibility of $F_{\id_*, \id_*}$ with the associators is automatic, whereas the compatibility of $F_0$ and $F_{\id_*, \id_*}$ with the left and right unitors is equivalent to the condition that $\fcomp = f_0^{-1}$.
				
				On the level of 2-morphisms, for $n \in \mathbb{Z}$, we have $F(n) = [f(n)]_{1_G}$ for an element $f(n) \in A$. Since $F$ respects vertical composition of 2-morphisms, we must have that $F(n) = F(1+1+...+1) = F(1)\circ F(1) \circ... \circ F(1) = F(1)^n$, so that $f(n) = f(1)^n \in A$, and the value of $F$ on 2-morphisms is completely determined by the single value $f(1) \in A$. The property of respecting horizontal composition is then automatically satisfied.
			\end{proof}
			
			\subsection{\texorpdfstring{Holonomy on $S^2$: 1-morphisms}{Holonomy on s2: 1-morphisms}}
			\begin{prop}\label{prop: holonomy 1-morphisms}
				Let $E, F \in \Fun(*\sq \pitwoS, * \sq \cG)$ be two functors, determined by the pairs $$(\ecomp, e(1)), (\fcomp, f(1)) \in A^2$$ under the bijection of Proposition \ref{prop: holonomy objects}. Then there is a bijection 
				\begin{align*}
					\{ \text{Natural transformations } E \Rightarrow F \} \EquivTo & \{g \in G \mid f(1) = g \cdot e(1) \}\\
					\sigma \mapsto & \sigma_*.
				\end{align*}
			\end{prop}
			
			\begin{proof} 
				Let $\sigma \colon E \Rightarrow F$ be a natural transformation. For the unique object in $* \sq \pitwoS$, $\sigma$ gives an isomorphism $\sigma_* \colon E(*) \to F(*)$ in $* \sq \cG$. Since $E(*)$ and $F(*)$ are both the unique object $\ptG$ of this bicategory, $\sigma_*$ is just an object of $\cG$, i.e. an element $\sigmaG \in G$. Next, for the unique 1-morphism $\id_*$ in $* \sq \pitwoS$, naturality of $\sigma$ is the data of a 2-morphism $\sigma_{\id_*}$ in $* \sq G$ making the following diagram commute: 
				\begin{align*}
					\begin{tikzpicture}[baseline=(basepoint)];
						\node (A) at (0,0) {$E(*) = \ptG$};
						\node (B) at (4,0) {$\ptG = E(*)$};
						\node (C) at (0,-2) {$F(*) = \ptG$};
						\node (D) at (4,-2) {$\ptG = F(*)$.};
						\draw[->] (A) to node [above] {$E(\id_*) = 1_G$}(B);
						\draw[->] ([xshift=-0.6cm]B.south) to node [right] {$\sigma_*$} ([xshift=-0.6cm]D.north);
						\draw[->] ([xshift=0.55cm]A.south) to node [left] {$\sigma_*$}([xshift=0.55cm]C.north);
						\draw[->] (C) to node [below] {$F(\id_*) = 1_G$} (D);
						\path (0,-.5) coordinate (basepoint);
					\end{tikzpicture}
				\end{align*}
				That is, $\sigma_{\id_*} = [s]_{\sigmaG} \colon \sigma_* \circ 1_G  \to 1_G \circ \sigma_*$ for some $s \in A$. 
				
				This 2-morphism is required to satisfy three compatibility conditions: it should be compatible with the action of $E$ and $F$ on 2-morphisms $\id_* \to \id_*$ in $* \sq \pitwoS$; it should be compatible with composition of $\id_*$ with itself; and it should be compatible with the unitors of $*\sq \pitwoS$ and $* \sq \cG$. 
				
				The first condition is that for any $n \in \mathbb{Z}$, giving a 2-morphism $\id_* \to \id_*$ in $* \sq \pitwoS$, the following diagram of 2-morphisms in $* \sq \cG$ (or equivalently, of morphisms in $\cG$) must commute:
				\begin{align*}
					\begin{tikzpicture}[baseline=(basepoint)];
						\node (A) at (0,0) {$\sigma_* \circ E(\id_*)$};
						\node (B) at (4,0) {$\sigma_* \circ E(\id_*)$};
						\node (C) at (0, -2) {$F(\id_*) \circ \sigma_*$};
						\node (D) at (4, -2) {$F(\id_*) \circ \sigma_*.$};
						\draw[->] (A) to node [above] {$\id_{\sigma_*} \wedge E(n)$} (B);
						\draw[->] (B) to node [right] {$\sigma_{id_*}$} (D);
						\draw[->] (A) to node [left] {$\sigma_{id_*}$} (C);
						\draw[->] (C) to node [below] {$F(n) \wedge \id_{\sigma_*}$} (D);
					\end{tikzpicture}
				\end{align*}
				This translates to the following equation in $A$:
				\begin{align*}
					s (\sigmaG \cdot (e(1)^n) = f(1)^n s \text{ for all } n \in \mathbb{Z}.
				\end{align*}
				Using that $A$ is abelian and that the action of $g$ is a group homomorphism, this is equivalent to the single condition on $g$,
				\begin{align}\label{eq: condition 1 for holonomy 1-morphisms}
					f(1) = g \cdot e(1).
				\end{align}
				The second condition is the commutativity of the following octagon:
				\begin{align*}
					\begin{tikzpicture}[
						xscale=1.8, 
						auto,
						every node/.style={font=\footnotesize},
						obj/.style={inner sep=3pt} 
						]
						\node[obj] (A1) at (112.5:3cm) {$(F(\id_*) \circ \sigma_*) \circ E(\id_*)$};
						\node[obj] (A2) at (67.5:3cm)  {$F(\id_*) \circ (\sigma_* \circ E(\id_*))$};
						\node[obj] (A3) at (22.5:3cm)  {$F(\id_*) \circ (F(\id_*) \circ \sigma_*)$};
						\node[obj] (A4) at (337.5:3cm) {$(F(\id_*) \circ F(\id_*)) \circ \sigma_*$};
						\node[obj] (A5) at (292.5:3cm) {$F(\id_*) \circ \sigma_*.$};
						\node[obj] (A6) at (247.5:3cm) {$\sigma_* \circ E(\id_*)$};
						\node[obj] (A7) at (202.5:3cm) {$\sigma_* \circ (E(\id_*) \circ E(\id_*))$};
						\node[obj] (A8) at (157.5:3cm) {$(\sigma_* \circ E(\id_*)) \circ E(\id_*)$};
						\draw[->] (A1) -- node {$\alpha_{1_G, \sigmaG, 1_G}$} (A2);
						\draw[->] (A2) -- node {$\id_{F(\id_*)} \wedge \sigma_{\id_*}$} (A3);
						\draw[->] (A3) -- node {$\alpha_{1_G, 1_G, \sigmaG}^{-1}$} (A4);
						\draw[->] (A4) -- node {$F_{\id_*, \id_*} \wedge \id_{\sigma_*}$} (A5);
						\draw[->] (A6) -- node {$\sigma_{\id_*}$} (A5);
						\draw[->] (A7) -- node {$\id_{\sigma_*} \wedge E_{\id_*, \id_*}$} (A6);
						\draw[->] (A8) -- node {$\alpha_{\sigmaG, 1_G, 1_G}$} (A7);
						\draw[->] (A8) -- node {$\sigma_{\id_*} \wedge \id_{E(\id_*)}$} (A1);
					\end{tikzpicture}
				\end{align*}
				Interpreting this diagram in $A$ and using that $\alpha$ is normalized yields the following equation:
				\begin{align*}
					\fcomp s^2 = s(\sigmaG\cdot \ecomp)
				\end{align*}
				Because $A$ is abelian, this simplifies to
				\begin{align}\label{eq: condition 2 for holonomy 2-morphisms}
					s= \frac{g \cdot \ecomp}{\fcomp},
				\end{align}
				so that $s$ is uniquely determined once $g$ has been fixed. 
				
				The third condition is that the following diagram must commute:
				\begin{align*}
					\begin{tikzpicture}[
						xscale=1.8,
						auto,
						every node/.style={font=\small},
						obj/.style={inner sep=3pt} 
						]
						\node[obj] (A1) at (90:2cm)  {$\sigma_* \circ \id_{E(*)}$};
						\node[obj] (A2) at (18:2cm)  {$\sigma_* \circ E(\id_*)$};
						\node[obj] (A3) at (306:2cm)   {$F(\id_*) \circ \sigma_*.$};
						\node[obj] (A4) at (234:2cm) {$\id_{F(*)} \circ \sigma_*$};
						\node[obj] (A5) at (162:2cm) {$\sigma_*$};
						\draw[->] (A1) -- node {$\id_{\sigma_*}\wedge E_0$} (A2);
						\draw[->] (A2) -- node {$\sigma_{\id_*}$} (A3);
						\draw[->] (A4) -- node {$F_0 \wedge \id_{\sigma_*}$} (A3);
						\draw[->] (A4) -- node {$r_{\sigma_*}$} (A5);
						\draw[->] (A5) -- node {$l_{\sigma_*}^{-1}$} (A1);
					\end{tikzpicture}
				\end{align*}
				Since the left and right unitors of $\cG$ are trivial, this translates to the condition
				\begin{align*}
					f_0 = s(g\cdot e_0) \in A,
				\end{align*}
				which is equivalent to Equation \eqref{eq: condition 2 for holonomy 2-morphisms} because $\ecomp = (e_0)^{-1}$.
				
				In summary, the condition on $f(1) = g \cdot e(1)$ in the statement of the proposition comes from equation \eqref{eq: condition 1 for holonomy 1-morphisms}, while equation \eqref{eq: condition 2 for holonomy 2-morphisms} tells us that $\sigma$ is completely determined by $\sigma_*$. Since there are no other conditions on $\sigma_*$, we have the desired bijection.
			\end{proof}
			
			\begin{rmk}\label{rmk: holonomy 1-morphisms vs stabilizer}
				In particular, we see that $(\ecomp, e(1))$ and $(\fcomp, f(1))$ are isomorphic if and only if $e(1)$ and $f(1)$ are in the same orbit for the action of $G$ on $A$, and in that case, the set of isomorphisms is non-canonically in bijection with the stabilizer $\Stab_G(e(1))$. A choice of bijection is determined by a choice of $g\in G$ such that $f(1) = g \cdot e(1)$. If we consider the set of automorphisms of the functor $E$, the natural choice $g=1_G$ gives us a canonical identification of the set of automorphisms with $\Stab_G(e(1))$. To compare the (higher) group structures, we first need to better understand the composition of natural transformations in the holonomy bicategory.
			\end{rmk}
			
			We claim that, under the bijection of Proposition \ref{prop: holonomy 1-morphisms}, the composition of 1-morphisms corresponds to multiplication of elements in $G$. More precisely, we have the following proposition. 
			
			\begin{prop}\label{prop: composition of holonomy 1-morphisms}
				Let $D, E, F \in \Fun(*\sq \pitwoS, * \sq \cG)$ be three functors, determined by the pairs $$(\dcomp, d(1)), (\ecomp, e(1)), (\fcomp, f(1)) \in A^2$$ under the bijection of Proposition \ref{prop: holonomy objects}. Let $\sigma \colon E \Rightarrow F$, $\sigma' \colon D \Rightarrow E$ be two natural transformations, corresponding to $\sigmaG, \tauG \in G$ under the bijection of Proposition \ref{prop: holonomy 1-morphisms}. Then $\sigma \circ \sigma'$ corresponds to the element $gg' \in G$ under this bijection. 
			\end{prop}
			
			\begin{proof}
				Recall that the composition $\sigma \circ \sigma'$ of natural transformations is defined as follows:
				
				For the unique object $*$ of $* \sq \pitwoS$, we have the 1-morphism
				\begin{align*}
					(\sigma \circ \sigma')_*: D(*) \xrightarrow{\sigma'_*} E(*) \xrightarrow{\sigma_*} F(*),
				\end{align*}
				i.e. $(\sigma \circ \sigma')_* = \sigmaG \tauG \in G$. We note that since $g\cdot e(1) = f(1)$ and $g' \cdot d(1) = e(1)$, we do indeed have $(gg') \cdot d(1) = f(1)$, as required. This completely determines the behaviour of composition under the bijection of Proposition \ref{prop: holonomy 1-morphisms}, but for completeness, let us recall also the definition of $(\sigma \circ \sigma')_{\id_*}$. 
				
				Indeed, we have
				\begin{align*}
					(\sigma \circ \sigma')_{\id_*} \colon (\sigma \circ \sigma')_* \circ 1_G \to 1_G \circ (\sigma \circ \sigma')_*
				\end{align*}
				given by the following composition
				\begin{align*}
					(\sigma_* \circ \sigma_*') \circ 1_G & \xrightarrow{\alpha_{g,g',1_G}} \sigma_* \circ (\sigma'_* \circ 1_G) \xrightarrow{\id_{\sigma_*} \wedge \sigma'_{\id_*}} \sigma_* \circ (1_G \circ \sigma'_*)\\
					& \quad \xrightarrow{\alpha^{-1}_{g, 1_G, g'}} (\sigma_* \circ 1_G) \circ \sigma'_*  \xrightarrow{\sigma_{\id_*}\wedge \id_{\sigma'_*}} (1_G) \circ \sigma_*) \circ \sigma'_*\\
					& \quad \xrightarrow{\alpha_{1_G,g,g'}} 1_G \circ (\sigma_* \circ \sigma'_*).
				\end{align*}
				That is, $(\sigma \circ \sigma')_{\id_*} = [s(\sigmaG \cdot s')]_{\sigmaG\tauG}$. It is easy to check that this satisfies the appropriate version of equation \eqref{eq: condition 2 for holonomy 2-morphisms}. \end{proof}
			
			\begin{rmk}\label{rmk: holonomy 1-morphisms vs stabilizer group structure}
				In particular, the canonical bijection between the set of automorphisms of a functor $E$ and $\Stab_G(e(1))$ is a group isomorphism. Of course, we should really think of the automorphisms as the objects in a monoidal category or 2-group, rather than a group; we will consider the categorical structure in Proposition \ref{prop: holonomy 2-morphisms} and Corollary \ref{cor: holonomy automorphism 2-groups}.
			\end{rmk}
			
			\subsection{\texorpdfstring{Holonomy on $S^2$: 2-morphisms}{Holonomy on s2: 2-morphisms}}
			Finally, let us consider 2-morphisms in the bicategory $\Fun(*\sq \pi_{\le 2}(S^2), * \sq \cG)$, also known as \emph{modifications} between natural transformations. 
			\begin{prop}\label{prop: holonomy 2-morphisms}
				Let $E, F \in \Fun(*\sq \pitwoS, * \sq \cG)$ be two functors, determined by the pairs $$(\ecomp, e(1)), (\fcomp, f(1)) \in A^2$$ under the bijection of Proposition \ref{prop: holonomy objects}. Let $\sigma, \sigma' \colon E \Rightarrow F$ be two natural transformations, corresponding to $\sigmaG, \tauG \in G $ as in Proposition \ref{prop: holonomy 1-morphisms}. If $\sigmaG \ne \tauG$, then there are no modifications between $\sigma$ and $\sigma'$. If $\sigmaG = \tauG$, then we have a bijection
				\begin{align}\label{eq: holonomy 2-morphisms bijection}
					\{\text{modifications } \sigma \to \sigma' \} & \to A\\
					M & \mapsto m \text{ such that } M_* =[m]_{\sigmaG}.\nonumber
				\end{align}
			\end{prop}
			
			\begin{proof}
				Consider a modification $M \colon \sigma \to \sigma'$. Explicitly, $m$ is the data of a 2-morphism
				\begin{align*}
					M_* \colon \sigma_* \to \sigma'_*
				\end{align*}
				in $* \sq \cG$. In other words, we must have that $\sigmaG = \tauG \in G$, and $M_* = [m]_\sigmaG$ for some $m \in A$. It remains to show that there are no conditions on this element $m \in A$. Indeed, in principle the 2-morphism $M_*$ is required to be compatible with the 2-morphisms $\sigma_{\id_*}$ and $\sigma'_{\id_*}$, in that the following diagram must commute. 
				\begin{align*}
					\begin{tikzpicture}[baseline=(basepoint)];
						\node (A) at (0,0) {$\sigma_* \circ E(\id_*)$};
						\node (B) at (4,0) {$\sigma'_* \circ E(\id_*)$};
						\node (C) at (0, -2) {$F(\id_*) \circ \sigma_*$};
						\node (D) at (4, -2) {$F(\id_*) \circ \sigma'_*.$};
						\draw[->] (A) to node [above] {$M_* \wedge \id_{1_G}$} (B);
						\draw[->] (B) to node [right] {$\sigma'_{id_*}$} (D);
						\draw[->] (A) to node [left] {$\sigma_{id_*}$} (C);
						\draw[->] (C) to node [below] {$\id_{1_G}\wedge M_*$} (D);
					\end{tikzpicture}
				\end{align*}
				In $A$, this condition becomes
				\begin{align*}
					\tauG m = m \sigmaG,
				\end{align*}
				which simplifies to $\tauG = \sigmaG$, since $A$ is abelian. We conclude that there is no condition on $m \in A$, as desired.
			\end{proof}
			
			\begin{prop}\label{prop: holonomy 2-morphisms vertical composition}
				Under the bijection of Proposition \ref{prop: holonomy 2-morphisms}, vertical composition of modifications corresponds to multiplication in $A$. That is, the bijection \eqref{eq: holonomy 2-morphisms bijection} is an isomorphism of groups. 
			\end{prop}
			
			\begin{proof}
				Given $E, F$ as in the statement of Proposition \ref{prop: holonomy 2-morphisms}, let $\sigma, \sigma', \sigma''$ be three natural transformations between them, corresponding to $g,g',g'''\in G$. Suppose that we have modifications $M \colon \sigma \to \sigma', M' \colon \sigma' \to \sigma''$. By Proposition \ref{prop: holonomy 2-morphisms}, we know that $g=g'=g'' \in G$, and that $M_*=[m]_g, M'_*=[m']_g$ for some $m, m' \in A$. The composed modification $M' \circ M$ is determined by the 2-morphism $(M' \circ M)_* = M'_* \circ M_* = [m'm]_g$, as claimed. 
			\end{proof}
			
			\begin{prop}\label{prop: holonomy 2-morphisms horizontal composition} 
				Under the bijection of Proposition \ref{prop: holonomy 2-morphisms}, horizontal composition admits the following natural description: suppose we have $D,E,F$ three functors classified by pairs $(a_1,a_2), (b_1,b_2),(c_1,c_2) \in A^2$,  and two natural transformations $\sigma\colon D \Rightarrow E$ and $\sigma' \colon E \to F$ be natural transformations, determined by $g, g' \in G$.  Let $M \colon \sigma \to \sigma, M' \colon \sigma' \to \sigma'$ be modifications, determined by $M_* = [m]_g$ and $M'_* = [m']_{g'}$ for some $m, m' \in A$, depicted in the diagram on the left.  
				\begin{align}\label{eq: holonomy horizontal composition diagram}
					\begin{tikzpicture}[scale=.8, every node/.style={font=\scriptsize}]
						\node (A) at (0,0) {$(a_1,a_2)$};
						\node (B) at (3,0) {$(b_1,g \cdot a_2)$};
						\node (C) at (6.5,0) {$(c_1,g'g \cdot a_2)$};
						\node (D) at (8.5,0) {$(a_1,a_2)$};
						\node (E) at (13.5, 0) {$(c_1,g'g \cdot a_2).$};
						\draw[->] (A) to[bend left=40] node[name=f, above] {$g$} (B);
						\draw[->] (A) to[bend right=40] node[name=fprime, below] {$g$} (B);
						\draw[->] (B) to[bend left=40] node[name=g, above] {$g'$} (C);
						\draw[->] (B) to[bend right=40] node[name=gprime, below] {$g'$} (C);
						\draw[->] (D) to[bend left=35] node[name=h, above] {$g'g$} (E);
						\draw[->] (D) to [bend right=35] node[name=hprime, below] {$g'g$} (E);
						\draw[->, double, double distance=1.2pt, shorten <=5pt, shorten >=5pt] (f) -- node[right] {$\! m$} (fprime);
						\draw[->, double, double distance=1.2pt, shorten <=5pt, shorten >=5pt] (g) -- node[right] {$\!m'$} (gprime);
						\draw[->, double, double distance=1.2pt, shorten <=5pt, shorten >=5pt] (h) -- node[right] {$\!m' \wedge m\!$} (hprime);
					\end{tikzpicture}
				\end{align}
				Then the horizontal composition of the modifications, as shown in the diagram on the right, corresponds to the element $m'(g'\cdot m) \in A$. 
			\end{prop}
			
			\begin{proof}
				Recall that the horizontal composition $M' \wedge M \colon \sigma' \circ \sigma \to \sigma' \circ \sigma$ is determined by $(M' \wedge M)_* \colon (\sigma' \circ \sigma)_* \to (\sigma' \circ \sigma)_*$, given by
				\begin{align*}
					g'g \xrightarrow{[m']_{g'} \wedge [m]_{g}} g'g.
				\end{align*}
				In other words, $(M' \wedge M)_*=[m'(g' \cdot m)]_{g'g}$ as claimed. 
			\end{proof}
			
			\subsection{\texorpdfstring{A simple description of the holonomy bicategory on $S^2$}{A simple description of the holonomy bicategory on s2}}
			Motivated by the results of the previous three subsections, we make the following definition.
			
			\begin{defn} \label{def: bicategory Y}
				Let $G,A, \alpha$ be as in Example \ref{ex: G, A, alpha}. Let $\cY$ be the bicategory with 
				\begin{itemize}
					\item objects $A \times A$;
					\item morphisms $\Hom_\cY((a_1,a_2),(a'_1, a'_2)) = \left\{g \in G \mid a'_2 = g \cdot a_2\right\}$;
					\item 2-morphisms $2\Hom_\cY(g,g') = \left\{ \begin{array}{cc}
						A &  \text {if }g = g'\\
						\varnothing & \text{else;}
					\end{array}
					\right.$
					\item composition of 1-morphisms given by multiplication in $G$;
					\item associator given by $\alpha$;
					\item identity 1-morphisms given by $\id_{(a_1,a_2)} = 1_G$;
					\item left and right unitors given by $1_A$;
					\item vertical composition of 2-morphisms given by multiplication in $A$;
					\item identity 2-morphisms given by $\id_{g} = 1_A$;
					\item horizontal composition of 2-morphisms given by the diagram \eqref{eq: holonomy horizontal composition diagram} with $m' \wedge m = m'(g' \cdot m) \in A=2\Hom_\cY(g'g, g'g).$
				\end{itemize}
			\end{defn}
			\begin{thm}\label{thm: holonomy and Y}
				The data $\cY$ forms a bicategory, and we have a canonical isomorphism 
				\begin{align*}
					\Fun(*\sq \pi_{\le 2}(S^2), * \sq \cG) \cong \mathcal{Y}.
				\end{align*}
			\end{thm}
			\begin{rmk}
				Here we really mean an isomorphism, not just an equivalence, since we have bijections on the level of objects, morphisms, and 2-morphisms, strictly respecting the categorical structures. Therefore the bicategory axioms on $\cY$ follow immediately from the existence of this isomorphism. 
			\end{rmk}
			\begin{proof}
				Proposition \ref{prop: holonomy objects} gives us a bijection on the level of objects. Proposition \ref{prop: holonomy 1-morphisms} gives us a bijection on the level of 1-morphisms, which respects composition and preserves identity morphisms by Proposition \ref{prop: composition of holonomy 1-morphisms}. Proposition \ref{prop: holonomy 2-morphisms} gives us a bijection on the level of 2-morphisms, which  respects vertical composition by Proposition \ref{prop: holonomy 2-morphisms vertical composition} and horizontal composition by Proposition \ref{prop: holonomy 2-morphisms horizontal composition}. It remains to compare the left and right unitors and the associators. However, these structure maps for the bicategory $\Fun(* \sq \pitwoS, * \sq \cG)$ are induced from those of $* \sq \cG$, which in turn are induced from those of $\cG$: i.e. the unitors are trivial, and the associator is given by $\alpha$, just as in $\cY$. 
			\end{proof}
			
			\begin{cor}\label{cor: holonomy automorphism 2-groups}
				Let $E \in \Fun(* \sq \pitwoS, * \sq \cG)$ be a functor, determined by a pair $(\ecomp, e(1))$. Let $\alpha_\vert$ denote the restriction of the cocycle $\alpha$ to the subgroup $\Stab_G(e(1)) \subset G$. Then we have canonical and strict isomorphisms of 2-groups 
				\begin{align*}
					\Aut(E) \cong \Aut_\cY(\ecomp, e(1)) \cong \cG(\Stab_{G}(e(1)), A, \alpha_\vert).
				\end{align*}
			\end{cor}
			
			\begin{proof}
				This follows from Theorem \ref{thm: holonomy and Y} together with Remark \ref{rmk: holonomy 1-morphisms vs stabilizer group structure}.
			\end{proof}
			
			\section{Proof of the main theorem}\label{sec: proof}
			Recall the skeletal bicategory $\cS$, defined using a choice of orbit representatives $C \subset A$.
			\begin{thm}\label{thm: S and Y}
				The inclusion 
				\begin{align*}
					C & \to A \times A \\
					a & \mapsto (1_A, a)
				\end{align*}
				extends to a strict functor 
				\begin{align*}
					F \colon \cS \to \cY
				\end{align*}
				giving an equivalence of bicategories. In particular, $\cS$ is a bicategorical skeleton of $\cY$. 
			\end{thm}
			
			\begin{proof}
				
				
				
				
				We will extend the assignment $a \mapsto (1_A,a)$ to a functor $F$ as follows. We need equivalences of $\Hom$-categories 
				\begin{align*}
					\Hom_\cS(a,a') \to \Hom_\cY((1_A, a), (1_A, a')).
				\end{align*}
				for all $a, a' \in C$, compatible with composition. We note that, by definition, the category $\Hom_\cS(a,a')$ is non-empty if and only if $a = a'$. On the other hand, by Remark \ref{rmk: holonomy 1-morphisms vs stabilizer}, the category $\Hom_\cY((1_A, a), (1_A, a'))$ is non-empty if and only if $a$ and $a'$ are in the same $G$-orbit. Since $a, a' \in C$, this happens if and only if $a=a'$ as well. 
				
				So it suffices to provide equivalences of the automorphism 2-groups 
				\begin{align*}
					\Aut_\cS(a) \to \Aut_\cY((1_G,a))
				\end{align*}
				for each $a \in C$. But we have already observed (Corollary \ref{cor: holonomy automorphism 2-groups} and Lemma \ref{lemma: S automorphism 2-group}) that both sides are strictly isomorphic to $\cG(\Stab_G(a), A, \alpha_\vert)$, so we define our functor $F$ to act as the identity on the level of 1-morphisms and 2-morphisms. 
				
				It remains to check that $F$ is essentially surjective, that is that any object $(a_1,a_2)$ of $\cY$ is isomorphic to one of the form $(1_A,a)$ for some $a \in C$. Indeed, we take $a \in C$ to be the representative of the orbit $G\cdot a_2$, so that there exists at least one $g \in G$ such that $a_2 = g \cdot a$. By Definition \ref{def: bicategory Y}, $g$ gives an isomorphism in $\Hom_\cY((1_A, a),(a_1,a_2))$, as desired. 
			\end{proof}
			
			We now have all the ingredients in place to provide an explicit equivalence of bicategories, proving that principal $\cG$-bundles on $S^2$ are indeed classified by holonomy. 
			\begin{thm}
				We have an equivalence of bicategories
				\begin{align*}
					\Hol \colon \Bun_{\cG}(S^2) \to \Fun(* \sq \pitwoS, * \sq \cG).
				\end{align*}
				Under the isomorphism $\Fun(* \sq \pitwoS, * \sq \cG) \cong \cY$ of Theorem \ref{thm: holonomy and Y}, this equivalence is given by
				\begin{description}
					\item[On objects] $(\rho, \gamma) \mapsto (1_A, a_{(\rho, \gamma)})$, where $a_{(\rho, \gamma)} = \frac{\gamma_{123}\gamma_{134}}{(\rho_{12}\cdot \gamma_{234})\gamma_{124}}\alpha(\rho_{12}, \rho_{23}, \rho_{34})$;
					\item[On 1-morphisms] $(t, \eta) \mapsto t_1$;
					\item[On 2-morphisms] $\omega \mapsto \omega_1$. 
					
				\end{description}
			\end{thm}
			
			\begin{proof}
				We need to show that the above assignments give a well-defined functor $\Hol$. The fact that $\Hol$ is an equivalence of bicategories will then follow from the commutativity of the triangle, and the fact that $\Psi$ and $F$ are equivalences, by Theorems \ref{thm: S and Bun} and \ref{thm: S and Y}, respectively.
				\begin{align*}
					\begin{tikzpicture}[baseline=(basepoint)];
						\node (A) at (0,0) {$\Bun_\cG(S^2)$};
						\node (B) at (4,0) {$\cY.$};
						\node (C) at (2, -1) {$\cS$};
						\draw[->] (C) to node [below left] {$\Psi$} (A);
						\draw[->] (C) to node [below right] {$F$} (B);
						\draw[->] (A) to node [above] {$\Hol$} (B);
					\end{tikzpicture}
				\end{align*}
				Let us check that $\Hol$ is indeed a functor. First, let $(t, \eta) \colon (\rho, \gamma) \to (\rho', \gamma')$ be a 1-morphism of principal $\cG$-bundles. We note that the concordance condition \eqref{eq: concordance condition} is exactly the condition
				\begin{align*}
					\frac{a_{(\rho', \gamma')}}{t_1 \cdot a_{(\rho, \gamma)}} = 1,
				\end{align*}
				so that $a_{(\rho', \gamma')} = t_1 \cdot a_{(\rho, \gamma)}$, and $t_1$ is indeed a 1-morphism $(1_A, a_{(\rho, \gamma)}) \to (1_A, a_{(\rho', \gamma')})$ in $\cY$. It is also easy to see from the definition of composition in $\Bun_\cG(S^2)$ (Definition \ref{def: composition of Cech 1-morphisms}) that $\Hol$ strictly preserves composition. It is even more straightforward to see that $\Hol$ sends 2-morphisms to 2-morphisms in a way compatible with vertical and horizontal compositions, and that intertwines the associators of the two bicategories. 
				
				Finally, it is a straightforward calculation to see that the diagram commutes, so that $\Hol$ is an equivalence, as claimed. 
			\end{proof}

			\appendix
			\section{\texorpdfstring{Comparison: a geometric definition of $\cG$-bundles}{Comparaison : a geometric definition of principal G-bundles}}\label{sec: appendix compare}
			
			There are many alternative definitions of principal 2-groups in the literature, many of which are known to be equivalent (at least under some conditions on the 2-group $\cG$). Let sketch a comparison of our Definition \ref{def: Cech data objects} with what we'll call the \emph{geometric} definition \cite{Bartels, Schommer-Pries}.  We provide this definition here because it is transparently a generalization of the classic definition of a principal $G$-bundle over $X$ as a manifold $P \to X$ equipped with an action of $G$ satisfying certain properties. However, the definition takes place in the bicategory of smooth stacks, and we do not take this as our primary definition because we wish the main body of our article to be accessible to readers without background in stacks. 
			
			\begin{conventions}
				The bicategory of smooth stacks is equivalent to the bicategory $\Bibun$, whose objects are Lie groupoids, whose 1-morphisms are (right principal) bibundles, and whose 2-morphisms are isomorphisms of bibundles \cite{Pronk, Lerman}. We recall that a smooth functor between Lie groupoids gives rise to a bibundle, giving a fully faithful embedding of the category of smooth functors and smooth natural transformations between any two fixed Lie groupoids, and the hom-category in $\Bibun$ between the same two Lie groupoids. In fact, the specific examples of morphisms we will use in this section all happen to be smooth functors, rather than general bibundles.
				
				In this setting, we can define not just abstract 2-groups, but \emph{smooth} 2-groups: they are group objects in the bicategory of smooth stacks (or equivalently in $\Bibun$). 
			\end{conventions}

			\begin{defn}
				Let $\cG$ be a smooth 2-group. Let $X$ be a smooth manifold. A \emph{principal $\cG$-bundle} over $X$ is a stack $\cP \xrightarrow{\pi} X$ equipped with a right action of $\cG$ on $\cP$, such that $\pi$ is $\cG$-invariant and $\cP \to X$ is locally trivial. 
			\end{defn}
			Let us expand a little on several of the terms used in this definition. First of all, an \emph{action} of $\cG$ on $\cP$ consists of a morphism of stacks
			\begin{align*}
				\Act \colon \cP \times \cG \to \cG
			\end{align*}
			together with 2-morphisms evincing the usual axioms of an action. That is, there are 2-morphisms making the following diagrams commute. 
			
			\newcommand{\mult}{\mathsf{mult}}
			\begin{align*}
				\begin{tikzpicture}[baseline=(basepoint)];
					\node (A) at (0,0) {$\cP \times \cG \times \cG$};
					\node (B) at (4,0) {$\cP \times \cG$};
					\node (C) at (0, -2) {$\cP \times \cG$};
					\node (D) at (4,-2) {$\cG;$};
					\draw[->] (A) to node [above] {$\Act \times \id_\cG$} (B);
					\draw[->] (A) to node [left] {$\id_\cP \times \mult $} (C);
					\draw[->] (B) to node [right] {$\Act$} (D);
					\draw[->] (C) to node [below] {$\Act$} (D);
					\draw[->, double, double distance=1.2pt, shorten <=35pt, shorten >=35pt] (B) to (C);
					\node (E) at (6,0) {$\cP$};
					\node (F) at (10,0) {$\cP\times \cG$};
					\node (G) at (8,-2) {$\cP.$};
					\draw[->] (E) to node [above] {$\id_\cP \times 1_\cG$} (F);
					\draw[->] (F) to node [below right] {$\Act$} (G);
					\draw[->] (E) to node [below left] {$\id_\cP$} (G);
					\draw[->, double, double distance=1.2pt, shorten <=30pt, shorten >=30pt] (9.5,0) to (7,-1.5);
				\end{tikzpicture}
			\end{align*}
			Secondly, the condition of \emph{local triviality} means that there exists a surjective submersion $u \colon Y \to X$ such that the pullback $u^*\cP$ is trivial: there is an equivariant isomorphism $u^*\cP \EquivTo Y \times \cG$ over $Y$. Note that equivariance is additional data, not just a condition on the morphism. 
			
			\begin{defn}
				Let $(\pi_i \colon \cP_i \to X, \Act_i), i=1,2$ be two principal $\cG$-bundles over $X$. A \emph{morphism} of principal $\cG$-bundles from $(\pi_1 \colon \cP_1 \to X, \Act_1)$ to $(\pi_2 \colon \cP_2 \to X, \Act_2)$ is given by a $\cG$-equivariant morphism $(T, \epsilon) \colon \cP_1 \to \cP_2$ over $X$. 
			\end{defn}
			
			\begin{defn}\label{def: Schommer-Pries bundles}
				The bicategory $\Bun_\cG^{\text{geo}}(X)$ of principal $\cG$-bundles over $X$ is the bicategory with objects principal $\cG$-bundles, 1-morphisms given by (ref), and 2-morphisms given by natural transformations of equivariant morphisms. 
			\end{defn}
			
			The bicategory of flat $\cG$-bundles is simply the bicategory $\Bun_{\cG^\text{disc}}^{\text{geo}}$ of bundles for the discretization $\cG^\text{disc}$ of the smooth 2-group $\cG$. If the 2-group is discrete to begin with (such as our 2-group $\cG(G, A, \alpha)$ equipped with the discrete topology), all bundles are automatically flat. 
			
			Let us now sketch the comparison between the geometric definition of principal $\cG$-bundles and \v Cech data for a good cover $\bU= \{U_i\}$ of $X$ for discrete $\cG = \cG(G, A, \alpha)$. We note that we follow the ideas of \cite{BCMNP} for the variant of the \v Cech bicategory given there (see Remark \ref{rmk: compare BCMNP} and \cite[Lemma 4.16]{BCMNP}); we provide the key ideas of the proof here to indicate (a) that any principal bundle admits \v Cech data over the good cover $\bU$ (Construction \ref{constr: extract Cech data}) and (b) that the construction extends to our setting of non-trivial $G$-action on $A$ (Construction \ref{constr: gluing}). 
			\begin{constr}\label{constr: extract Cech data}
				Let $\cP \to X$ be a flat principal $\cG$-bundle on $X$. For any $i \in I$, $U_i$ is contractible, so the restriction of $\cP$ to $U_i$ (i.e. the homotopy fibre product $\cP \times_X U_i$) admits a section. Equivalently, this gives a trivialization $\varphi_i \colon U_i \times \cG \to \cP \times_X U_i$. Over $U_{ij}$, we obtain the following sequence of $\cG$-equivariant isomorphisms: 
				\begin{align*}
					U_{ij} \times \cG \xrightarrow{\varphi_j} (\cP \times_X U_j) \times_{U_j}U_{ij} \cong (\cP \times_X U_i) \times_{U_i}U_{ij} \xrightarrow{\varphi_i^{-1}} U_{ij} \times \cG.
				\end{align*}
				But a $\cG$-equivariant automorphism of $U_{ij} \times \cG$ is equivalent to a bibundle $\hat{\rho}_{ij}\colon U_{ij} \to \cG$. Since $U_{ij}$ is contractible (or empty), this bibundle is always equivalent to a smooth functor, and hence given by a locally constant map $U_{ij} \to G$ (i.e. an element of $G$, if $(i,j) \in \nonempty{2}$). 
				
				The second piece of \v Cech data comes from the natural isomorphism of bibundles over $U_{ijk}$ for all $(i,j,k) \in \nonempty{3}$:
				\begin{align*}
					\varphi_i^{-1} \circ \varphi_k \EquivTo (\varphi_i^{-1} \circ \varphi_j) \circ (\varphi_j^{-1} \circ \varphi_k).
				\end{align*}
				This is equivalent to an isomorphism of bibundles $\hat{\rho}_{ik}$ and $\hat{\rho}_{ij} \hat{\rho}_{jk}$ from $U_{ijk}$ to $\cG$. Since we have chosen identifications of both of these bibundles with smooth functors, and smooth functors form a full subcategory of $\Bibun$, this natural isomorphism is equivalent to a smooth natural transformation of the smooth functors $U_{ijk} \to \cG$, which is just a smooth function $U_{ijk} \to A$ (i.e. an element $\gamma_{ijk} \in A$). It is not hard to check that the pair $(\rho, \gamma)$ satisfies the conditions of Definition \ref{def: Cech data objects}.
			\end{constr}
			
			\begin{constr}\label{constr: gluing}
				Let $(\rho, \gamma)$ be an object as in Definition \ref{def: Cech data objects}. We define $\cP$ to be the Lie groupoid
				\begin{align*}
					\begin{array}{c}
						\bigsqcup_{ij} U_{ij} \times G \times A  \\
						s\downarrow\downarrow t \\
						\bigsqcup_i U_i \times G
					\end{array}
				\end{align*}
				where $s(y, g,a) = (y, \rho_{ij}g) \in U_i \times G$ and $t(y,g,a)=(y,g) \in U_j \times G$ for $(y,g,a) \in U_{ij} \times G \times A$, and with composition given by 
				\begin{align*}
					(y,g,b) \circ (y, \rho_{jk}g, a) = \left(y, g, \frac{a(\rho_{ij}\cdot b)}{\gamma_{ijk}\alpha(\rho_{ij}, \rho_{jk}, g)}\right).
				\end{align*}
				Property \eqref{eq: cocycle condition for rho} ensures that the composition has the correct source and target, while property \eqref{eq: condition for gamma} ensures associativity. Meanwhile, \eqref{eq: normalisation condition for eta} implies that the identity morphism for $(y,g) \in U_i \times G$ is given by $(y,g,1_A) \in U_{ii} \times G \times A$. 
				
				The Lie groupoid $\cP$ has a natural smooth functor to $X$, and admits the following right action of $\cG$:
				\begin{align*}
					\begin{array}{ccc}
						\bigsqcup_{ij} U_{ij} \times G \times A \times G \times A & \xrightarrow{(y,g,a,h,b) \mapsto (y,gh, a(g \cdot b))} & \bigsqcup_{ij} U_{ij} \times G \times A \\
						s\downarrow\downarrow t  & & s\downarrow\downarrow t \\
						\bigsqcup_i U_i \times G \times G & \xrightarrow{\quad (y,g,h) \mapsto (y,gh) \quad}& \bigsqcup_i U_i \times G. 
					\end{array}
				\end{align*}
				It is straightforward to check that this bundle trivializes over $\bigsqcup_{i \in I} U_i$, and hence forms a principal $\cG$-bundle. 
			\end{constr}
			
			\section{\texorpdfstring{Explicit tables of the data of $\Bun_\cG(S^2)$}{Explicit tables of the data of bunG}}\label{sec: appendix table}
			
			\newcommand{\action}{\mathbin{\boldsymbol{\cdot}} }
			
			\begin{table}[H]\label{table: objects rho}
				\renewcommand{\arraystretch}{1.25}
				\centering
				\begin{tabular}{|C *{4}{|C}|}
					\hline
					i & j=1 &j=2 & j=3 & j=4  \\
					\hline
					1 & 1 & \rho_{12} & \rho_{12}\rho_{23}   & \rho_{12}\rho_{23}\rho_{34} \\
					2 & \rho_{12}^{-1} & 1 & \rho_{23} & \rho_{23}\rho_{34} \\
					3 & \rho_{23}^{-1}\rho_{12}^{-1} & \rho_{23}^{-1} & 1 & \rho_{34} \\
					4 & \rho_{34}^{-1}\rho_{23}^{-1}\rho_{12}^{-1} & \rho_{34}^{-1}\rho_{23}^{-1} & \rho_{34}^{-1} & 1\\[0.5ex]
					\hline
				\end{tabular}
				\caption{Evaluations of $\rho_{ij}$ when $\rho_{12}$, $\rho_{23}$ and $\rho_{34}$ are fixed.}
				\label{tab: rho_ij possible}
			\end{table}
			
			\begin{table}[H] \label{table: objects gamma}
				\centering
				\renewcommand{\arraystretch}{1.75}
				\resizebox{\columnwidth}{!}{
					\begin{tabular}{*{6}{|C}|}
						\hline
						\multicolumn{2}{|C|}{\gamma_{ijk}} & \multirow{2}{*}{$k=1$} & \multirow{2}{*}{$k=2$} & \multirow{2}{*}{$k=3$} & \multirow{2}{*}{$k=4$} \\
						\cline{1-2}
						i & j &  &  &  &  \\
						\hline 
						1 & 2 & \gamma_{121} & 1 & \gamma_{123} & \gamma_{124} \\
						
						1 & 3 & \gamma_{131} & \frac{\pr{\rho_{12}\action \gamma_{232}\alpha_{2132}}\alpha_{1213}}{\gamma_{123}\alpha_{1212}}& 1 &\gamma_{134}\\[1ex]
						
						1 & 4 & \gamma_{141} & \frac{\pr{\rho_{12}\action \gamma_{242}\alpha_{2142}}\alpha_{1214}}{\gamma_{124}\alpha_{1212}} & \frac{\pr{\rho_{13}\action \gamma_{343}\alpha_{3143}}\alpha_{1314}}{\gamma_{134}\alpha_{1313}}  & 1 \\[1ex]
						
						2 & 1 & 1 & \rho_{21}\action \gamma_{121}\alpha_{1212} & \frac{\pr{\rho_{21}\action \gamma_{121}\alpha_{1212}}\alpha_{2123}}{\rho_{21}\action \gamma_{123}} & \frac{\pr{\rho_{21}\action \gamma_{121}\alpha_{1212}}\alpha_{2124}}{\rho_{21}\action \gamma_{124}} \\[1ex]
						
						2 & 3 & \frac{\rho_{21}\action \gamma_{123}\gamma_{131}\alpha_{1231}}{\rho_{21}\action \gamma_{121}} & \gamma_{232} & 1 & \gamma_{234} \\[1ex]
						
						2 & 4 & \frac{\rho_{21} \action \gamma_{124}\gamma_{141}\alpha_{1241}}{\rho_{21} \action \gamma_{121}} & \gamma_{242} & \frac{\pr{\rho_{23}\action \gamma_{343}\alpha_{3243}}\alpha_{2324}}{\gamma_{234}\alpha_{2323}} & 1\\[1ex]
						
						3 & 1 & 1 & \frac{\rho_{31}\action \gamma_{123}\gamma_{131}\alpha_{1313}}{\pr{\rho_{32}\action \gamma_{232}\alpha_{2323}}\alpha_{3123}} & \rho_{31} \action \gamma_{131}\alpha_{1313} 
						& \frac{\pr{\rho_{31}\action \gamma_{131}\alpha_{1313}}\alpha_{3134}}{\rho_{31}\action \gamma_{134}}\\[1ex]
						
						3 & 2 & \frac{\pr{\rho_{32} \action \gamma_{232} \alpha_{2321}} \pr{\rho_{31} \action \gamma_{121}}}{\rho_{31} \action \gamma_{123}\gamma_{131} \alpha_{1231}} 
						& 1  & \rho_{32} \action \gamma_{232}\alpha_{2323} & \frac{\pr{\rho_{32}\action \gamma_{232}\alpha_{2323}}\alpha_{3234}}{\rho_{32}\action \gamma_{234}}\\[1ex]
						
						3 & 4 & \frac{\rho_{31} \action \gamma_{134}\gamma_{141}\alpha_{1341}}{\rho_{31} \action \gamma_{131}} 
						& \frac{\rho_{32} \action \gamma_{234}\gamma_{242}\alpha_{2342}}{\rho_{32} \action \gamma_{232}} 
						& \gamma_{343} & 1 \\[1ex]
						
						4 & 1 & 1 & \frac{\rho_{41}\action \gamma_{124}\gamma_{141}\alpha_{1414}}{\pr{\rho_{42}\action \gamma_{242}\alpha_{2424}}\alpha_{4124}}  & \frac{\rho_{41}\action \gamma_{134}\gamma_{141}\alpha_{1414}}{\pr{\rho_{43}\action \gamma_{343}\alpha_{3434}}\alpha_{4134}}  &\rho_{41} \action \gamma_{141}\alpha_{1414} \\[1ex]
						
						4 & 2 & \frac{\pr{\rho_{42} \action \gamma_{242} \alpha_{2421}} \pr{\rho_{41} \action \gamma_{121}}}{\rho_{41} \action \gamma_{124}\gamma_{141} \alpha_{1241}} 
						& 1 & \frac{\rho_{42}\action \gamma_{234}\gamma_{242}\alpha_{2424}}{\pr{\rho_{43}\action \gamma_{343}\alpha_{3434}}\alpha_{4234}}  & \rho_{42} \action \gamma_{242}\alpha_{2424} \\[1ex]
						
						4 & 3 & \frac{\pr{\rho_{43} \action \gamma_{343} \alpha_{3431}} \pr{\rho_{41} \action \gamma_{131}}}{\rho_{41} \action \gamma_{134}\gamma_{141} \alpha_{1341}} 
						& \frac{\pr{\rho_{43} \action \gamma_{343} \alpha_{3432}} \pr{\rho_{42} \action \gamma_{232}}}{\rho_{42} \action \gamma_{234}\gamma_{242} \alpha_{2342}} 
						& 1 & \rho_{43} \action \gamma_{343}\alpha_{3434}\\[2ex] 
						\hline
				\end{tabular}}
				\renewcommand{\arraystretch}{1}
				\caption{Evaluations of $\gamma_{ijk}$ when $\gamma_{121}$, $\gamma_{131}$, $\gamma_{141}$, $ \gamma_{232}$, $\gamma_{242}$, $\gamma_{343}$, $\gamma_{123}$, $\gamma_{124}$, $\gamma_{134}$ and $\gamma_{234}$ are fixed. Convention : $\rho_{ij} \action \gamma_{klmn} \alpha_{opqr} = \rho_{ij} \action (\gamma_{klmn} \alpha_{opqr})$ and $\alpha_{abcd} = \alpha(\rho_{ab}, \rho_{bc}, \rho_{cd})$.  }
				\label{tab: gamma possibles}
			\end{table}
			
			\begin{table}[H] \label{table: 1-morphisms t}
				\renewcommand{\arraystretch}{1.5}
				\centering
				\begin{tabular}{|C|C C C C|}
					\hline 
					i & 1 & 2 & 3 & 4 \\ \hline
					t_i & t_1 & \rho'_{21}t_1\rho_{12} & \rho'_{31}t_1\rho_{13} & \rho'_{41}t_1\rho_{14}\\[0.5ex]
					\hline
				\end{tabular}
				\renewcommand{\arraystretch}{1}
				\caption{Evaluations of $t_i$ when $t_1$ is fixed.}
				\label{tab: t_i}
			\end{table}
			
			\begin{table}[H] \label{table: 1-morphisms eta}
				\centering
				\renewcommand{\arraystretch}{2}
				\resizebox{\columnwidth}{!}{
					\begin{tabular}{|C *{4}{|C}|}
						\hline
						i & j=1 &j=2 & j=3 & j=4  \\[0.5ex]
						\hline
						1 & 1 & \eta_{12} & \frac{\eta_{12}\pr{\rho'_{12}\action \eta_{23}}\pr{t_1\action \gamma_{123}} \alpha\pr{\rho'_{12},t_2,\rho_{23}}}{\gamma'_{123} \alpha\pr{t_1,\rho_{12},\rho_{23}} \alpha\pr{\rho'_{12},\rho'_{23},t_3}} 
						& \frac{\eta_{12}\pr{\rho'_{12}\action \eta_{24}}\pr{t_1\action \gamma_{124}}\, \alpha\pr{\rho'_{12},t_2,\rho_{24}}}{\gamma'_{124} \, \alpha\pr{t_1,\rho_{12},\rho_{24}} \, \alpha\pr{\rho'_{12},\rho'_{24},t_4}} \\
						
						2 & \rho_{21}\action \frac{\gamma'_{121} \alpha(t_1,\rho_{12},\rho_{21}) \alpha(\rho'_{12},\rho'_{21},t_1)}{\eta_{12}\gamma_{121} \alpha(\rho'_{12},t_2,\rho_{21})} 
						& 1 & \eta_{23} 
						& \frac{\eta_{23}\pr{\rho'_{23}\action \eta_{34}}\pr{t_2\action \gamma_{234}} \alpha\pr{\rho'_{23},t_3,\rho_{34}}}{\gamma'_{234} \alpha\pr{t_2,\rho_{23},\rho_{34}} \alpha\pr{\rho'_{23},\rho'_{34},t_4}} \\
						
						3 &  \rho_{31}\action \frac{\gamma'_{131} \alpha(t_1,\rho_{13},\rho_{31}) \alpha(\rho'_{13},\rho'_{31},t_1)}{\eta_{13}\gamma_{131} \alpha(\rho'_{13},t_3,\rho_{31})}
						& \rho_{32}\action \frac{\gamma'_{232} \alpha(t_2,\rho_{23},\rho_{32}) \alpha(\rho'_{23},\rho'_{32},t_2)}{\eta_{23}\gamma_{232} \alpha(\rho'_{23},t_3,\rho_{32})}
						& 1 & \eta_{34} \\
						
						4 & \rho_{41}\action \frac{\gamma'_{141} \alpha(t_1,\rho_{14},\rho_{41}) \alpha(\rho'_{14},\rho'_{41},t_1)}{\eta_{14}\gamma_{141} \alpha(\rho'_{14},t_4,\rho_{41})}
						& \rho_{42}\action \frac{\gamma'_{242} \alpha(t_2,\rho_{24},\rho_{42}) \alpha(\rho'_{24},\rho'_{42},t_2)}{\eta_{24}\gamma_{242} \alpha(\rho'_{24},t_4,\rho_{42})} 
						& \rho_{43}\action \frac{\gamma'_{343} \alpha(t_3,\rho_{34},\rho_{43}) \alpha(\rho'_{34},\rho'_{43},t_3)}{\eta_{34}\gamma_{343} \alpha(\rho'_{34},t_4,\rho_{43})} & 1 \\[1.25ex]
						\hline
				\end{tabular}}
				\renewcommand{\arraystretch}{1}
				\caption{Evaluations of $\eta_{ij}$ when $\eta_{12}$, $\eta_{23}$ and $\eta_{34}$ are fixed. }
				\label{tab: eta_ij possible}
			\end{table}
			
			\begin{table}[H]\label{table: 2-morphisms}
				\renewcommand{\arraystretch}{2}
				\centering
				\begin{tabular}{|C|C C C C|}
					\hline 
					i & 1 & 2 & 3 & 4 \\[0.5ex] \hline 
					\omega_i & \omega_1 & \frac{\pr{\rho'_{21}\action\omega_1}\eta_{21}}{\eta'_{21}} 
					& \frac{\pr{\rho'_{31}\action \omega_1} \eta_{31}}{\eta'_{31}} 
					& \frac{\pr{\rho'_{41} \action \omega_1} \eta_{41}}{\eta'_{41}}\\[1.25ex]
					\hline
				\end{tabular}
				\renewcommand{\arraystretch}{1} 
				\caption{Evaluations of $\omega_i$ when $\omega_1$ is fixed.}
				\label{tab: omega possibles}
			\end{table}

			\printbibliography[heading=bibintoc, title={References}]

%
%
	
	
	\vspace{5mm}
	\noindent
	Emily Cliff, Yves Lolelo, and Sandrine St-Cyr \\
	Département de mathématiques \\
	Université de Sherbrooke \\
2500, boulevard de l'Université \\Sherbrooke (Québec) Canada

	
\end{document}